\documentclass[11pt]{amsart}
\usepackage{a4wide}
\usepackage{amsmath}
\usepackage[utf8]{inputenc}
\usepackage{amssymb}
\usepackage{amsopn}
\usepackage{epsfig}
\usepackage{amsfonts}
\usepackage{latexsym}
\usepackage{graphicx, subfigure}
\usepackage{epstopdf}
\usepackage{calrsfs}
\usepackage{tikz}
\usepackage{enumerate}

\usepackage{dsfont}

 \usetikzlibrary{arrows,automata}
\DeclareMathAlphabet{\pazocal}{OMS}{zplm}{m}{n}
\newtheorem{theorem}{Theorem}[section]
\newtheorem{lemma}[theorem]{Lemma}
\newtheorem{proposition}[theorem]{Proposition}

\theoremstyle{definition}

\newtheorem{example}[theorem]{Example}

\theoremstyle{remark}

\numberwithin{equation}{section}

\newcommand{\R}{\ensuremath{\mathbb{R}}}

\newcommand{\ep}{\varepsilon}
\newcommand{\f}{\infty}

\newcommand{\om}{\omega}

\newcommand{\lu}{L\"{u}roth}
\newcommand{\lue}{L\"{u}roth expansion}
\newcommand{\lut}{L\"{u}roth transformation}
\newcommand{\alue}{alternating L\"{u}roth expansion}

\newcommand{\galue}{generalised $\alpha$-L\"{u}roth expansion}

\newcommand{\lf}{\lfloor}
\newcommand{\rf}{\rfloor}

\begin{document}

\title{Rational approximation with generalised $\alpha$-L\"{u}roth expansions}
\author[Y.Huang]{Yan Huang}
\address[Y. Huang]{College of Mathematics and Statistics, Chongqing University, 401331, Chongqing, P.R.China. \&
Mathematisch Instituut, Leiden University, Niels Bohrweg 1, 2333CA Leiden, The Netherlands
}
\email{{yanhuangyh@126.com}}

\author[C. Kalle]{Charlene Kalle}
\address[C. Kalle]{Mathematisch Instituut, Leiden University, Niels Bohrweg 1, 2333CA Leiden, The Netherlands}
\email{kallecccj@math.leidenuniv.nl}
%\dedicatory{}
%\date{\today}
%\begin{frontmatter}
\subjclass[2010]{11K55, 11A67, 37A10}

\begin{abstract}
For a fixed $\alpha$, each real number $x \in (0,1)$ can be represented by many different generalised $\alpha$-L\"uroth expansions. Each such expansion produces for the number $x$ a sequence of rational approximations $(\frac{p_n}{q_n})_{n \ge 1}$. In this paper we study the corresponding approximation coefficients $(\theta_n(x))_{n \ge 1}$, which are given by
\[ \theta_n (x): = q_n \left|x-\frac{p_n}{q_n}\right|.\]
We give the cumulative distribution function and the expected average value of the $\theta_n$ and we identify which generalised $\alpha$-L\"uroth expansion gives the best approximation properties. We also analyse the structure of the set $\mathcal M_\alpha$ of possible values that the expected average value of $\theta_n$ can take, thus answering a question from \cite{Barrionuevo-Burton-Dajani-Kraaikamp-1994}.
\end{abstract}

\keywords{{\galue}; cumulative distribution function; approximation coefficient; Cantor set.}
\maketitle

\section{Introduction}\label{sec-introduction}

\emph{\lue s} were first introduced by {\lu} \cite{Luroth-1883} in 1883 and are expressions for real numbers $x\in[0,1]$ of the form
\begin{equation}\label{eq:def-lu-e}
x=\frac{1}{d_1}+\frac{1}{d_1(d_1-1)d_2}+\ldots =\sum_{n\geq 1}(d_n-1)\prod_{i=1}^n\frac{1}{d_i(d_i-1)}, \quad d_n\in\mathbb{N}_{\geq2}\cup\{\infty\}, \, n \ge 1.
\end{equation}
Since the numbers
\begin{equation*}
\frac{p_{n,L}}{q_{n,L}}:=\sum_{k=1}^n(d_k-1)\prod_{i=1}^k\frac{1}{d_i(d_i-1)},\quad n\in\mathbb{N},
\end{equation*}
give a sequence of rationals converging to the number $x$, one can wonder about the rational approximation properties of \lue s. The authors of \cite{Barreira-Iommi-2009} proved that for Lebesgue $a.e.$ $x\in[0,1]$,
$$\lim_{n\to\infty}-\frac{1}{n}\log \left|x-\frac{p_{n,L}}{q_{n,L}}\right|=\sum_{d=2}^\infty\frac{\log (d(d-1))}{d(d-1))}.$$
They further gave a multifractal analysis of the speed of convergence and showed that the range of possible values of this rate is $[\log 2, \infty)$.

\vskip .2cm
Another way to express the quality of the approximations is via the limiting behaviour of the approximation coefficients
\[ \theta_n^L(x):=q_{n,L}\left|x-\frac{p_{n,L}}{q_{n,L}}\right|,\quad n\in\mathbb{N},\]
where $q_{n,L}=d_n\prod_{i=1}^{n-1}d_i(d_i-1)$ and $\prod_{i=1}^0 d_i(d_i-1)=1$. In \cite{Barrionuevo-Burton-Dajani-Kraaikamp-1994} it was shown that for Lebesgue a.e.~$x \in [0,1]$, $\lim_{N \to \infty} \frac1N \sum_{n=1}^N \theta_n^L (x) = \frac12 (\zeta(2)-1)$, where $\zeta$ is the Riemann zeta function.

\vskip .2cm
In 1990 Kalpazidou, Knopfmacher and Knopfmacher introduced \emph{alternating \lue s} in \cite{Kalpazidou-Knopfmacher-Knopfmacher-1990}, where the terms in the summation from (\ref{eq:def-lu-e}) alternate in sign. It was subsequently proved in \cite{Barrionuevo-Burton-Dajani-Kraaikamp-1994} that the \alue s outperform the \lue s in terms of approximation properties. Since then both {\lu} and \alue s have been extensively studied and several generalisations have been considered. In \cite{Barrionuevo-Burton-Dajani-Kraaikamp-1994} a very general set-up is presented, of which the \emph{$\alpha$-L\"uroth  expansions} from \cite{Sara-2011} are a specific case. The properties of such expansions were then further discussed in \cite{Marc-Sara-Bernd-2012}. A random version of generalised \lue s was given in \cite{Kalle-Maggioni-2022}.

\vskip .2cm
In this article we consider {\em generalised $\alpha$-L\"uroth expansions}, of which L\"uroth expansions and alternating L\"uroth expansions are specific examples. We define these generalised $\alpha$-L\"uroth expansions  by describing the dynamical algorithms for obtaining them. It was already observed by Jager and De Vroedt in \cite{Jager-Vroedt-1968} that \lue s can be obtained by dynamically iterating the so-called \emph{\lut} and the same holds for all other generalised L\"uroth expansions mentioned above, as shown in \cite{Barrionuevo-Burton-Dajani-Kraaikamp-1994}. The transformations given below are specific instances of the ones from \cite{Barrionuevo-Burton-Dajani-Kraaikamp-1994}.

\vskip .2cm
To each strictly decreasing sequence $(t_n)_{n \ge 1} \subseteq (0,1]$ satisfying $t_1 =1$ and $\lim_{n \to \infty} t_n=0$ we can associate a countable interval partition $\alpha=\{A_n:=(t_{n+1},t_n]:n\in\mathbb{N}\} \cup \{ A_\infty := \{0\} \}$ of $[0,1]$. We call them {\em $\alpha$-L\"uroth partitions} and use $\mathcal{P}$ to denote the class of such partitions. For each $\alpha \in \mathcal P$, set $a_n:=t_n-t_{n+1}$, so $a_n$ is the Lebesgue measure of $A_n$. \emph{Generalised $\alpha$-L\"uroth transformations} $T_\varepsilon = T_{\alpha,\varepsilon}:[0,1]\to [0,1]$, indexed by sequences $\varepsilon = (\varepsilon_n)_{n \ge 1} \in \{0,1\}^\mathbb N$, are then %obtained by taking on each interval $A_n$ an affine map with positive or negative slope such that $T_\varepsilon(A_n)$ has Lebesgue measure 1. The information on the sign of the slope is stored in the sequence $\ep\in\{0,1\}^\mathbb{N}$ as follows. For each $\alpha \in \mathcal P$ and $\ep\in\{0,1\}^\mathbb{N}$ let the {\galut} $T_\ep:[0,1]\to[0,1]$ be given
defined by setting
\[
\begin{split}
T_\varepsilon(x)=\left\{
\begin{array}{ll}
0, \quad & x\in A_\infty;\\
(x-t_{n+1})/a_n, \quad & x\in A_n, \,  n \neq \infty \, {\text{ and }} \, \ep_n=0;\\
(t_{n}-x)/a_n, \quad & x\in A_n, \,  n \neq \infty \, {\text{ and }}\, \ep_n=1.\\
\end{array}
\right.
\end{split}
\]
See Figure~\ref{f:tepsilon} for an illustration.

\begin{figure}[h]
\begin{tikzpicture}[scale=3]
\draw(-.01,0)node[below]{\footnotesize $0$}--(.21,0)node[below]{\footnotesize $t_4$}--(.5,0)node[below]{\footnotesize $t_3$}--(.65,0)node[below]{\footnotesize $t_2$}--(1,0)node[below]{\footnotesize$t_1=1$}--(1.01,0);
\draw(0,-.01)--(0,1)node[left]{\footnotesize $1$}--(0,1.01);
\draw[thick, green!60!black] (.65,0)--(1,1)(.5,1)--(.65,0)(.21,1)--(.5,0)(.17,1)--(.21,0)(.15,0)--(.17,1)(.1,1)--(.15,0)(.07,0)--(.1,1)(.05,0)--(.07,1)(.035,0)--(.05,1)(.03,1)--(.035,0)(.02,1)--(.03,0)(.01,0)--(.02,1)(.005,0)--(.01,1);
\draw[dotted] (0,1)--(1,1)--(1,0)(.65,0)--(.65,1)(.5,0)--(.5,1)(.21,0)--(.21,1)(.17,0)--(.17,1)(.15,0)--(.15,1);
\node[above] at (.82,1) {\footnotesize $A_1$};
\node[above] at (.57,1) {\footnotesize $A_2$};
\node[above] at (.35,1) {\footnotesize $A_3$};
\end{tikzpicture}
\caption{A transformation $T_{\alpha,\varepsilon}$ with $\varepsilon_1=0$, $\varepsilon_2=1$, $\varepsilon_3=1$, $\varepsilon_4=1$, $\varepsilon_5=0$, $\ldots$.}
\label{f:tepsilon}
\end{figure}
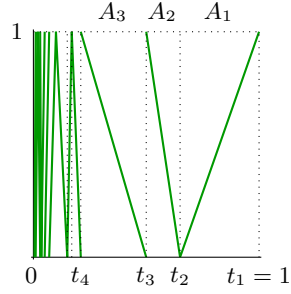

\vskip .2cm
If we set $\varepsilon_{\infty}=0$, then from $T_\varepsilon$ we can obtain \emph{digit sequences} $(d_n)_{n\geq1}$ and \emph{sign sequences} $(s_n)_{n\geq1}$ recursively as follows. For any $x\in[0,1]$ and $n \ge 1$ let $d_n=d_n(x)=k$ and $s_n =s_n(x)=\ep_k$ if $T_{\varepsilon}^{n-1}(x) \in A_k$. Note that with these definitions we get
\[ T_\varepsilon (x) = (-1)^{s_1} \frac{x-t_{d_1+1-s_1}}{a_{d_1}}.\]
Hence, inverting and iterating give
 \begin{equation}\label{eq:x-x-n}
 \begin{split}
 x=\ &t_{d_1+1-s_1}+(-1)^{s_1}a_{d_1} T_\varepsilon(x)\\
 = \ &t_{d_1+1-s_1}+(-1)^{s_1}a_{d_1}(t_{d_2+1-s_2}+(-1)^{s_2}a_{d_2}T_\varepsilon^2(x))\\
 = \ &t_{d_1+1-s_1}+(-1)^{s_1}a_{d_1}t_{d_2+1-s_2}+(-1)^{s_1+s_2}a_{d_1}a_{d_2}T_\varepsilon^2(x)\\
 \vdots \ & \\
 = \ &t_{d_1+1-s_1}+\dots+(-1)^{\sum_{i=1}^{n-1}s_i}t_{d_n+1-s_n} \prod_{i=1}^{n-1}a_{d_i}+
 (-1)^{\sum_{i=1}^{n}s_i}\prod_{i=1}^{n}a_{d_i}T_\varepsilon^n(x).
 \end{split}
 \end{equation}
Since $T_\varepsilon^n(x) \in [0,1]$ and $a_n \in (0,1)$ for each $n$, this converges and we get a number expansion of the form
\begin{equation*}%\label{eq:g-a-lu-ex-1}
x =\sum_{n\geq 1}(-1)^{\sum_{i=1}^{n-1}s_i} t_{d_n+1-s_n} \prod_{i=1}^{n-1}a_{d_i},
\end{equation*}
where $\sum_{i=1}^0s_i=0$ and $\prod_{i=1}^0a_{d_i}=1$. This is called a {\galue} of $x$. Let
\begin{equation}\label{eq:def-p/q}
\frac{p_n}{q_n}=\frac{p_{n, \alpha, \varepsilon}}{q_{n, \alpha, \varepsilon}}(x):=\sum_{k= 1}^n (-1)^{\sum_{i=1}^{k-1}s_i} t_{d_k+1-s_n} \prod_{i=1}^{k-1}a_{d_i},\quad n\in\mathbb{N},
\end{equation}
where $q_n=\frac{1}{t_{d_n+1-s_n}\prod_{i=1}^{n-1}a_{d_i}}$ and $\prod_{i=1}^0 a_{d_i}=1$, be the \emph{n-th approximation} of $x$. Then we define the \emph{approximation coefficients} $\theta_n^\ep$ by
\begin{equation}\label{eq:def-theta}
\theta_n^\ep=\theta_n^{\alpha,\ep}(x):=q_n\left|x-\frac{p_n}{q_n}\right|, \quad n\in\mathbb{N}.
\end{equation}
Since we are usually interested in the effect of the sequence $\varepsilon$ for a fixed partition $\alpha$, we will often keep the $\varepsilon$ and suppress the $\alpha$ in the notation when no confusion can arise.

\vskip .2cm
For the classical {\lu} partition $\alpha_L \in \mathcal P$, given by the sequence $(t_n)_{n \ge 1}$ with $t_n = \frac1n$, in \cite{Barrionuevo-Burton-Dajani-Kraaikamp-1994} Barrionuevo et al.~gave the cumulative distribution function of the $\theta_n^{\alpha_L,\ep}$.
\begin{theorem}[Theorem 4 in \cite{Barrionuevo-Burton-Dajani-Kraaikamp-1994}]\label{th:bbdk-F}
For any $\ep\in\{0,1\}^\mathbb{N}$, $z\in(0,1]$ and {Lebesgue} a.e. $x\in[0,1]$ the limit
$$\lim_{N\to\infty}\frac{\#\{1\leq n\leq N:\theta_n^{\alpha_L,\ep}(x)<z\}}{N}$$ exists and equals
\[
F_{\alpha_L,\ep}(z):=\sum_{k=2}^{\lf\frac{1}{z}\rf+1-\ep_{\lf\frac{1}{z}\rf}}\frac{z}{k-\ep_{k-1}}+\frac{1}{\lf\frac{1}{z}\rf+1-\ep_{\lf\frac{1}{z}\rf}}.
\]
\end{theorem}

We would like to remark that in \cite{Barrionuevo-Burton-Dajani-Kraaikamp-1994} the intervals in $\alpha_L$ are labeled starting with index 2, so $\alpha_L = \{ ( \frac1n, \frac1{n-1}] \, : \, n \ge 2\}\cup\{0\}$. Accordingly, the sequences $(d_n), (\varepsilon_n)$ and $(s_n)$ are also defined slightly differently. We have adapted the statement of Theorem~\ref{th:bbdk-F} above to fit our notation.

\vskip .2cm
%\medskip
If we write $\bar 0 $ and $\bar 1$ for the infinite sequences of 0's and 1's, respectively, then Theorem~\ref{th:bbdk-F} in particular gives that
$$F_L(z):=F_{\alpha_L, \bar 0} (z)=\sum_{k=2}^{\lf\frac{1}{z}\rf+1}\frac{z}{k}+\frac{1}{\lf\frac{1}{z}\rf+1}\quad and \quad F_A(z):=F_{\alpha_L,\bar 1}(z)=\sum_{k=2}^{\lf\frac{1}{z}\rf}\frac{z}{k-1}+\frac{1}{\lf\frac{1}{z}\rf}.$$
Moreover, it was shown in \cite{Barrionuevo-Burton-Dajani-Kraaikamp-1994} that $F_L(z)\leq F_{\alpha_L,\ep}(z)\leq F_A(z)$ for any $\ep\in\{0,1\}^\mathbb{N}$.

\vskip .2cm
In this paper we will prove that for any $\alpha\in\mathcal{P}$, $\ep\in\{0,1\}^\mathbb{N}$, $z\in(0,1]$ and Lebesgue a.e. $x\in[0,1]$ the limit
$$\lim_{N\to\infty}\frac{\#\{1\leq n\leq N:\theta_n^\ep(x)<z\}}{N}$$ exists. Denote this limit by $F_\ep(z)$ and let
\begin{equation}\label{eq:def-M-ep}
M_\ep = M_{\alpha, \varepsilon}:=\int_{[0,1]}(1-F_\ep )d\lambda,
\end{equation}
where we use $\lambda$ to denote the Lebesgue measure. Then $M_\varepsilon$ is the expected average value corresponding to the cumulative distribution function $F_\varepsilon$, so for Lebesgue $a.e.~x\in[0,1]$, $M_\ep=\lim_{n\to\infty}\frac{1}{n}\sum_{i=1}^{n}\theta_i^{\ep}(x)$ is the average value of the $\theta_n^\varepsilon(x)$. Let
$$\mathcal{M} = \mathcal M_\alpha :=\{M_\ep: \ep\in\{0,1\}^\mathbb{N}\}$$
be the set of all possible values that $M_\varepsilon$ can take. In \cite{Barrionuevo-Burton-Dajani-Kraaikamp-1994} it was obtained that $M_A := M_{\alpha_L,\bar 1} \leq M_{\alpha_L,\ep}\leq M_L:= M_{\alpha_L, \bar 0}$ for each $\ep\in\{0,1\}^\mathbb{N}$.
 %And $M_A:=\int_{[0,1]}(1-F_A )d\lambda=1-\frac{\zeta(2)}{2}$ and $M_L:=\int_{[0,1]}(1-F_L )d\lambda=\frac{\zeta(2)}{2}-\frac{1}{2}$.
%Then we have $M_A\leq M_\ep\leq M_L$.
Our first main result generalises Theorem \ref{th:bbdk-F} to arbitrary partitions $\alpha \in \mathcal P$. It also implies that the alternating map $T_{\alpha, \bar 1}$ always presents the best approximation properties and it gives some first information on the structure of $\mathcal{M}_\alpha$.

\begin{theorem}\label{thm:main-1}
Given a partition $\alpha\in\mathcal{P}$, for any $\ep\in\{0,1\}^\mathbb{N}$, $z\in(0,1]$ and Lebesgue a.e. $x\in[0,1]$ the limit
$$\lim_{N\to\infty}\frac{\#\{1\leq n\leq N:\theta_n^\ep(x)<z\}}{N}$$
exists and equals
\begin{equation*}
F_\ep(z):=\sum_{n\in N_1}a_n+\sum_{n\in N_2}t_{n+1-\ep_n}z,
\end{equation*}
with $$N_1:=\left\{n:\frac{a_n}{t_{n+1-\ep_n}}<z\right\}\quad and \quad N_2:=\left\{n:\frac{a_n}{t_{n+1-\ep_n}}\geq z\right\}.$$
Furthermore we have $M_{\bar 0}>M_{\bar 1}$ and
\begin{equation}\label{eq:structure-M-se-1}
\mathcal{M}=\bigcap_{n=0}^\f\bigcup_{\om\in\{0,1\}^n}[M_{\om \bar 1},M_{\om \bar 0}].
\end{equation}
\end{theorem}

After establishing \eqref{eq:structure-M-se-1} we give several results that describe the structure of $\mathcal{M}$ in more detail under various conditions. A partial summary of these results is given in the following theorem.

\begin{theorem}\label{thm:main-2}
Let $\alpha\in\mathcal{P}$ be such that $\rho := \lim_{n\to\infty}\frac{t_{n+1}}{t_n} \in [0,1]$ exists.
\begin{itemize}
  \item[(i)] If  $0\leq \rho<1/2$, then $\mathcal M$ is a Cantor set.
  \item[(ii)] If $1/2< \rho < 1$, then $\mathcal M$ is a finite union of closed intervals.
\end{itemize}
\end{theorem}

Additionally, in the case of (i) we have an expression for the Hausdorff dimension and packing dimension of $\mathcal M$ and in the case of (ii) we obtain a condition under which for each interior point $M$ of $\mathcal M$ there are uncountably many $\varepsilon \in \{0,1\}^\mathbb N$ with $M_\ep = M$. % if there exists an $N\geq 0$ such that $\frac{M_{\om \bar 0}-M_{\om \bar 1}}{M_{\om \bar 0}-M_{\om 0\bar 1}}<\frac{\sqrt{5}+1}{2}$ for all $\omega\in\{0,1\}^n$ and $n\geq N$.
Besides Theorem~\ref{thm:main-2}, we also identify several cases where $\mathcal M$ is a \emph{homogeneous Cantor set}.

\vskip .2cm
One particular case that is not covered by Theorem~\ref{thm:main-2} is when $\rho =1$. This includes the L\"uroth partition $\alpha_L$, for which the authors of \cite{Barrionuevo-Burton-Dajani-Kraaikamp-1994} asked the question whether the set $\mathcal{M}_{\alpha_L}$ could be a fractal set inside the interval $[M_A,M_L]$, see \cite[Remark~2, page 323]{Barrionuevo-Burton-Dajani-Kraaikamp-1994}. In this article we make some general remarks for $\rho=1$ and in particular for the L\"uroth partition $\alpha_L$ we obtain the following answer to the question from \cite{Barrionuevo-Burton-Dajani-Kraaikamp-1994}.
\begin{theorem}\label{thm:main-3}
The set $\mathcal{M}_{\alpha_L}$ is the union of eight disjoint closed subintervals of equal length. Furthermore, for any interior point $M$ of $\mathcal{M}_{\alpha_L}$ there exist uncountably many $\ep\in\{0,1\}^\mathbb{N}$ such that $M_\ep=M$.
\end{theorem}

In case $\lim_{n\to\infty}\frac{t_{n+1}}{t_n}$ exists and equals $1/2$, various outcomes are possible and we discuss some of them below. In particular we discuss the partition $\{ (2^{-n}, 2^{-(n-1)}]\}_{n \ge 1}\cup\{0\}$ for which $\frac{t_{n+1}}{t_n}=\frac12$ for each $n$. We also provide some examples where $\lim_{n\to\infty}\frac{t_{n+1}}{t_n}$ does not exist.

\vskip .2cm

The paper is organised as follows. In the next section we recall some basic notation and prove Theorem \ref{thm:main-1}. In Section \ref{se:3} we study the structure of $\mathcal{M}$ in more detail. We obtain various results that together give Theorem \ref{thm:main-2}. In this section we also give the fractal dimension of $\mathcal M$ in case $\mathcal M$ is a Cantor set. In Section \ref{sec:4} we prove Theorem \ref{thm:main-3} and we also present three examples for which the limit $\lim_{n\to\infty}\frac{t_{n+1}}{t_n}$ does not exist.

\section{Proof of Theorem \ref{thm:main-1}}\label{sec:2}
We first introduce some notation. Let $\{0,1\}^\mathbb{N}$ be the set of all infinite sequences $(\ep_i)=\ep_1\ep_2\ldots$ with $\ep_i\in\{0,1\}$ for all $i \ge 1$. By a $\emph{word}$ we mean a finite string of digits over $\{0,1\}$. Denote by $\{0,1\}^*$ the set of all finite words including the empty word $\epsilon$. For any $\omega\in\{0,1\}^*$ we denote by $|\omega|$ the length of $\omega$, where $|\epsilon|=0$. For any $N\geq 0$, $\{0,1\}^N$ denotes the set of all words of length $N$. For two words $\omega, \eta \in \{0,1\}^*$ we write $\omega \eta \in \{0,1\}^*$ for their concatenation. Similarly, for a word $\omega=\omega_1\ldots \omega_n$ and a sequence $\ep=\ep_1\ep_2\ldots$ we use $\omega\ep =\omega_1\ldots \omega_n\ep_1\ep_2\ldots \in \{0,1\}^\mathbb N$ to denote their concatenation. We use square brackets to denote {\em cylinder sets}, i.e., for a word $\omega = \omega_1 \ldots \omega_k \in \{0,1\}^k$, $k \ge 0$, we write
 \[ [\omega] = \{ \omega\varepsilon \, : \,  \varepsilon \in \{0,1\}^\mathbb N \}.\]
On $\{0,1\}^\mathbb{N}$ we can define the metric $d$ by setting $d(\varepsilon, \eta)=0$ if $\varepsilon = \eta$ and
\[ d(\varepsilon,\eta) = 2^{-\min\{ k \ge 1 \, : \, \varepsilon_k \neq \eta_k\}}\]
otherwise. Cylinder sets are open and closed in the corresponding topology.

\vskip .2cm
Since the generalised $\alpha$-L\"uroth transformations are a particular type of the generalised L\"uroth transformations from \cite{Barrionuevo-Burton-Dajani-Kraaikamp-1994}, it follows from \cite[Theorem 1]{Barrionuevo-Burton-Dajani-Kraaikamp-1994} that for any $\alpha\in\mathcal{P}$ the maps $T_{\alpha,\varepsilon}$ are measure preserving and ergodic with respect to Lebesgue measure, from which one can deduce among other things that for Lebesgue a.e.~$x\in [0,1]$ the sequence $\{T^n_{\alpha,\varepsilon}(x)\}_{n \ge 0}$ is uniformly distributed within the interval $[0,1]$ and that the digits $d_1,d_2,\ldots$ are independent identically distributed random variables with respect to Lebesgue measure and $\lambda(d_n=k)=a_k$. The next proposition gives the cumulative distribution function of the $\theta_n^\ep$ by using the above properties.

\begin{proposition}\label{prop:a-F-ep-z}
Given a partition $\alpha\in\mathcal{P}$, for any $\ep\in\{0,1\}^\mathbb{N}$ and $z\in(0,1]$ and for Lebesgue a.e.~$x\in[0,1]$ the limit
$$\lim_{N\to\infty}\frac{\#\{1\leq n\leq N:\theta_n^\ep(x)<z\}}{N}$$ exists and equals $F_\ep(z)$, where
\begin{equation}\label{eq:def-F}
F_\ep(z)=\sum_{k\in N_1}a_k+\sum_{k\in N_2}t_{k+1-\ep_k}z,
\end{equation}
with
$$N_1=\left\{k:\frac{a_k}{t_{k+1-\ep_k}}<z\right\}\quad and \quad N_2=\left\{k:\frac{a_k}{t_{k+1-\ep_k}}\geq z\right\}.$$
\end{proposition}

\begin{proof}
For any $x\in[0,1]$ we get by \eqref{eq:x-x-n} that
 \begin{align*}
 x=t_{d_1+1-s_1}+\dots+(-1)^{\sum_{i=1}^{n-1}s_i}\prod_{i=1}^{n-1}a_{d_i}t_{d_n+1-s_n}+
 (-1)^{\sum_{i=1}^{n-1}s_i}\prod_{i=1}^{n}a_{d_i}T_\ep^n(x),
 \end{align*}
 for any $n\geq 1$.
 Combining this with \eqref{eq:def-p/q} and \eqref{eq:def-theta} we have
 $$\theta_n^\ep=q_n\left|x-\frac{p_n}{q_n}\right|=\frac{a_{d_n}}{t_{d_n+1-s_n}}T_\ep^n(x),$$
 for all $n\geq 1$. Note that $\theta_n^\ep<z$ if and only if $T_\ep^n(x)<\frac{t_{d_n+1-s_n}}{a_{d_n}} z$. Since $T_\ep^n(x)\in [0,1]$, this automatically holds for all $n$ such that $d_n \in N_1$. If $d_n \in N_2$, remark that $T_\ep^n(x) < \frac{t_{d_n+1-s_n}}{a_{d_n}} z$ if and only if
 \[T_\ep^{n-1}(x) \in A_{d_n} \cap T^{-1}_\varepsilon \left(\left[ 0, \frac{t_{d_n+1-s_n}}{a_{d_n}} z \right)\right).\]
Hence,
 \[\begin{split}
\#\{1\leq n\leq N \, :\, & \theta_n^\ep(x)<z\}\\
=  \, \,  &\#\{1\leq n\leq N \, :\, T_\varepsilon^{n-1} (x) \in A_{k}, \, k \in N_1\}\\
& +\, \#\left\{1\leq n\leq N \, :\, T_\varepsilon^{n-1} (x) \in A_{k} \cap T^{-1}_\varepsilon \left(\left[ 0, \frac{t_{k+1-\varepsilon_k}}{a_k} z \right)\right), \, k \in N_2\right\}.
\end{split}\]
Using that $\lambda (A_k \cap T^{-1}_\varepsilon ([ 0, \frac{t_{k+1-\varepsilon_k}}{a_k} z ))) = a_k \cdot \frac{t_{k+1-\varepsilon_k}}{a_k}z = t_{k+1-\varepsilon_k}z$ for any $k \in N_2$ and that for Lebesgue a.e.~$x\in [0,1]$ the sequence $\{ T_\varepsilon^n(x)\}_{n \ge 0}$ is uniformly distributed in $[0,1]$, we see that for those $x$, the limits of both terms exist and equal
\[ \lim_{N \to \infty} \frac{\#\{1\leq n\leq N \, :\, T_\varepsilon^{n-1} (x) \in A_{k}, \, k \in N_1\}}{N} = \sum_{k\in N_1}\lambda( A_k) = \sum_{k \in N_1} a_k\]
and
\[ \lim_{N \to \infty} \frac{\#\left\{1\leq n\leq N \, :\, T_\varepsilon^{n-1} (x) \in A_{k} \cap T^{-1}_\varepsilon \left(\left[ 0, \frac{t_{k+1-\varepsilon_k}}{a_k} z \right)\right), \, k \in N_2\right\}}{N} = \sum_{k \in N_2} t_{k+1-\varepsilon_k}z.\]
This proves the result.
\end{proof}
\vskip .2cm
Given $\alpha\in\mathcal{P}$, by (\ref{eq:def-F}) for any $\ep\in\{0,1\}^\mathbb{N}$ we have $F_\ep(z)=\sum_{n\geq 1}f_n^{\ep_n}(z)$ where
\begin{align}\label{eq:F-sum-f}
\begin{split}
f_n^{\ep_n}(z):=\left\{
\begin{array}{ll}
a_n\quad &\text{if } \frac{a_n}{t_{n+1-\ep_n}}<z;\\
t_{n+1-\ep_n}z \quad &\text{if } \frac{a_n}{t_{n+1-\ep_n}}\geq z.\\
\end{array}
\right.
\end{split}
\end{align}
Then for any $\omega \in \{0,1\}^n$ by the dominated convergence theorem we have that
\begin{equation}\label{q:sizeI}
\int_{[0,1]} F_{\omega \bar 1} - F_{\omega \bar 0} \, d\lambda = \sum_{k \ge n+1} \int_{[0,1]} f_k^1-f_k^0 \, d\lambda.
\end{equation}
We will use this later.

\vskip .2cm
For each $\kappa \ge 0$, each $\varepsilon \in \{0,1\}^\mathbb N$ and $z\in (0,1]$ define the following sets of indices:
\[\begin{split}
N_1^{\varepsilon, \kappa} = N_1^{\varepsilon, \kappa} (z) :=\ & \left\{n \ge 1:\frac{a_{\kappa+n}}{t_{\kappa+n+1-\ep_n}}<z\right\},\\
N_2^{\varepsilon, \kappa} = N_2^{\varepsilon, \kappa} (z) :=\ & \left\{n \ge 1:\frac{a_{\kappa+n}}{t_{\kappa+n+1-\ep_n}}\ge z\right\}.
\end{split}\]
Note that for the sets $N_1$ and $N_2$ from Theorem~\ref{thm:main-1} we have $N_1=N_1^{\varepsilon,0}$ and $N_2=N_2^{\varepsilon,0}$. For any $\omega \in \{0,1\}^\kappa$, $\kappa \ge 0$, and any sequence $\varepsilon \in \{0,1\}^\mathbb N$ it holds that
\[F_{\omega \ep }(z)=\sum_{n=1}^\kappa f_n^{\omega_n}(z)+\sum_{n\in N_1^{\varepsilon,\kappa}}a_{\kappa+n}+\sum_{n\in N_2^{\varepsilon,\kappa}}t_{\kappa+n+1-\ep_n}z.\]
Since $f_n^{\ep_n}(z)=t_{n+1-\ep_n}z\leq a_n$ for any $n\in\mathbb{N}$ such that $\frac{a_n}{t_{n+1-\ep_n}}\geq z$, we see that $N_1^{\bar 0, \kappa} \subset N_1^{\ep, \kappa}$ and $N_2^{\ep, \kappa} \subset N_2^{\bar 0, \kappa}$. Since it also holds that $N_1^{\ep,\kappa}\setminus N_1^{\bar 0,\kappa} = N_2^{\bar 0,\kappa}\setminus N_2^{\ep,\kappa} $, we obtain
\begin{equation}\label{q:FepsilonminusF0} \begin{split}
F_{\omega \ep}(z)-F_{\omega \bar 0 }(z) =\ & \sum_{n \in N_1^{\ep,\kappa}\setminus N_1^{\bar 0,\kappa}} a_{\kappa+n} + \sum_{n \in N_2^{\varepsilon,\kappa}} t_{\kappa+n+1-\varepsilon_n}z - \sum_{n \in N_2^{\bar 0,\kappa}} t_{\kappa+n+1}z\\
=\ & \sum_{n \in N_1^{\ep,\kappa}\setminus N_1^{\bar 0,\kappa}} (a_{\kappa+n} - t_{\kappa+n+1}z) + \sum_{n \in N_2^{\varepsilon,\kappa}} (t_{\kappa+n+1-\varepsilon_n}- t_{\kappa+n+1})z.
\end{split}\end{equation}
Similarly,
\begin{equation}\label{q:F1minusFepsilon}
F_{\omega\bar 1}(z) - F_{\omega \ep}(z) =  \sum_{n \in N_1^{\bar 1,\kappa}\setminus N_1^{\ep,\kappa}} (a_{\kappa+n} - t_{\kappa+n+1-\varepsilon_n}z) + \sum_{n \in N_2^{\bar 1,\kappa}} (t_{\kappa+n}- t_{\kappa+n+1-\varepsilon_n})z.
\end{equation}

\vskip .2cm
Recall from \eqref{eq:def-M-ep} that we have defined
\[ M_\ep=\int_{[0,1]}(1-F_\ep )d\lambda\quad and \quad \mathcal{M}=\{M_\ep: \ep\in\{0,1\}^\mathbb{N}\}.\]
%For any $\ep\in\{0,1\}^\mathbb{N}$ by Proposition \ref{prop:a-F-ep-z} we have that for Lebesgue $a.e.~x\in[0,1]$
%$$M_\ep=\lim_{n\to\infty}\frac{1}{n}\sum_{i=1}^{n}\theta_i^\ep(x)=\mathbb{E}[\theta_n^\ep]=\int_{[0,1]}(1-F_\ep )d\lambda.$$
The next two lemmas state that for any $\alpha\in\mathcal{P}$ the quality of approximation of the \galue s produced by the function $T_{\alpha,\bar 1}$ is better than that of $T_{\alpha,\bar 0}$, thus extending the result from \cite[Corollary 2]{Barrionuevo-Burton-Dajani-Kraaikamp-1994}.

\begin{lemma}\label{le:M-0-leq-M-1}
For any $\alpha\in\mathcal{P}$, we have $M_{\bar 0}> M_{\bar 1}$.
\end{lemma}

\begin{proof}
We need to prove that $\int_{[0,1]}F_{\bar 1}-F_{\bar 0}\, d\lambda > 0$. By \eqref{q:FepsilonminusF0} with $\kappa=0$ and $\varepsilon = \bar 1$, we get
\[ F_{\bar 1}(z) - F_{\bar 0}(z) =  \sum_{n \in N_1^{\bar 1,0}\setminus N_1^{\bar 0,0}} (a_n - t_{n+1}z) + \sum_{n \in N_2^{\bar 1,0}} (t_n- t_{n+1})z.\]
Since $\frac{a_n}{t_{n+1}}\geq z$ for any $n\in N_1^{\bar 1,0}\setminus N_1^{\bar 0,0}$, for any $z \in (0,1]$ we obtain
\[ F_{\bar 1}(z)-F_{\bar 0}(z)\geq \sum_{n\in N_2^{\bar 1,0}}a_n z.\]
Set $z_0 = \frac{a_1}{2t_1}$ and fix a $\delta \in (0, \frac{a_1}{2t_1})$. Then for any $z \in [z_0-\delta, z_0+\delta]$ it holds that $1 \in N_2^{\bar 1,0}$ and hence,
$$\int_{[0,1]}\sum_{n\in N_2^{\bar 1,0}}a_n z\, d\lambda(z)\geq \int_{[z_0-\delta,z_0+\delta]} a_1z \, d\lambda(z)=  2a_1 z_0\delta >0.$$
So we proved that $M_{\bar 0}>M_{\bar 1}$.
\end{proof}

\begin{lemma}\label{le:M-o-0-leq-M-o-1}
Given an $\alpha\in\mathcal{P}$ and a word $\omega\in\{0,1\}^*$, for any $\ep\in\{0,1\}^\mathbb{N}$ we have $M_{\omega\bar 1}\leq M_{\omega \ep}\leq M_{\omega \bar 0}$.
\end{lemma}

\begin{proof}
Fix a $z \in (0,1]$, an $\ep\in\{0,1\}^\mathbb{N}$, a $\kappa \ge 0$ and a word $\omega\in\{0,1\}^\kappa$. As in the proof of Lemma \ref{le:M-0-leq-M-1} it is enough to show that $F_{\omega \bar 0}(z)\leq F_{\omega \ep}(z)\leq F_{\omega \bar 1}(z)$. The first inequality follows immediately from \eqref{q:FepsilonminusF0} by noting that $a_{\kappa+n} - t_{\kappa+n+1}z \ge 0$ for all $n \in N_1^{\varepsilon,\kappa}\setminus N_1^{\bar 0,\kappa}$ and that $t_{\kappa+n+1-\varepsilon_n}- t_{\kappa+n+1} \in \{0,a_{\kappa+n}\}$ for all $n$. The second inequality follows similarly from \eqref{q:F1minusFepsilon}.
\end{proof}

For each word $\omega\in\{0,1\}^*$ set
\begin{equation}\label{q:interval}
I_\omega:=[M_{\omega \bar 1},M_{\omega \bar 0}].
\end{equation}
 From Lemma \ref{le:M-0-leq-M-1} we know that $\lambda(I_\epsilon)>0$ and by Lemma \ref{le:M-o-0-leq-M-o-1} the intervals $I_\omega$ are well-defined for each $\omega$.

\begin{proof}[Proof of Theorem \ref{thm:main-1}]
For any $\varepsilon \in \{0,1\}^\mathbb N$ by Lemma~\ref{le:M-o-0-leq-M-o-1} it holds that $M_\varepsilon \in I_{\varepsilon_1 \cdots \varepsilon_n}$ for each $n \ge 1$. Hence,
\[ \mathcal{M} \subseteq \bigcap_{n=0}^\f\bigcup_{\om\in\{0,1\}^n}I_\om.\]
For the other direction, first note that for any $N\ge 0$ and any $\om\in\{0,1\}^N$ we have that $I_{\omega0} = [M_{\omega0\bar 1}, M_{\omega\bar 0}]$ and by Lemma \ref{le:M-o-0-leq-M-o-1} $M_{\omega0\bar 1} \ge M_{\omega\bar 1}$. Similarly, $I_{\omega1} = [M_{\omega\bar 1}, M_{\omega1\bar 0}]$ and $M_{\omega1\bar 0} \le M_{\omega
\bar 0}$. So, for $d=0,1$ the intervals $I_{\omega d}$ are both subintervals of $I_\om$ and $I_{\omega1}$ shares the left endpoint with $I_{\omega}$ and $I_{\omega0}$ shares the right endpoint with $I_{\omega}$. Now, let $M \in \bigcap_{n=0}^\f\bigcup_{\om\in\{0,1\}^n}I_\om$. Then for each $n \ge 1$ there is a collection of words $\omega \in \{0,1\}^n$, such that $M \in I_\omega$. Let $A_n$ be the union of the cylinder sets given by these words. Then $A_n$ is a closed subset of $\{0,1\}^\mathbb N$ and $A_{n+1} \subseteq A_n$ for each $n \ge 1$. Hence, $A = \lim_{n \to \infty} A_n$ is a closed set with $M \in I_A$, which means that there is a sequence $\varepsilon \in \{0,1\}^\mathbb N$ for which $M=M_\varepsilon$. Thus,
\[ \mathcal{M}=\bigcap_{n=0}^\f\bigcup_{\om\in\{0,1\}^n}I_\om.\]
Combined with Proposition \ref{prop:a-F-ep-z} and Lemma \ref{le:M-0-leq-M-1}, the proof is complete.
\end{proof}

\section{Structure of $\mathcal M$}\label{se:3}
In this section we will study $\mathcal{M}$ in detail and give the proof of Theorem \ref{thm:main-2}. By \eqref{eq:structure-M-se-1} $\mathcal{M}$ can be obtained by successively removing a sequence of open intervals from $[M_{\bar 1},M_{\bar 0}]$. %For each $N \ge 0$ we call elements of the set $ \{ I_\omega \, : \, \omega \in \{0,1\}^N \}$ an \emph{$N$-level basic interval}.
For an $I_\omega, \ \omega\in\{0,1\}^*$ one of two things can happen. Either $I_{\omega0}\cap I_{\omega1} = \emptyset$ or not, see Figure~\ref{fi:2}. Which case occurs, or whether $M_{\omega1\bar 0}<M_{\omega0\bar 1}$ or $M_{\omega1\bar 0}\geq M_{\omega0\bar 1}$, determines the structure of $\mathcal M$.

\begin{center}
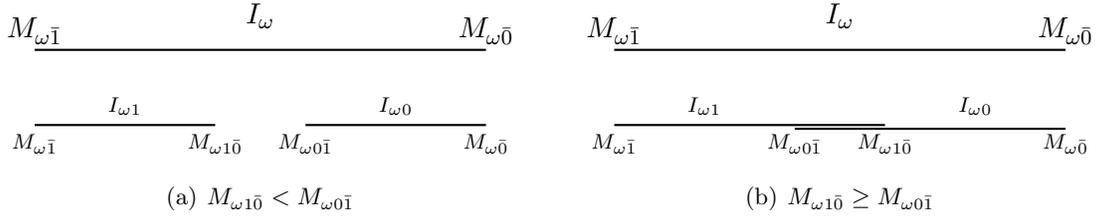
\begin{figure}[h]
\subfigure[$M_{\omega1\bar 0}<M_{\omega0\bar 1}$]{
\begin{tikzpicture}[xscale=6,yscale=50,axis/.style={very thick, ->}, important line/.style={thick}, dashed line/.style={dashed, thin},
    pile/.style={thick, ->, >=stealth', shorten <=2pt, shorten>=2pt},every node/.style={color=black} ]

    %1th
    \draw[important line] (0, 0)--(1, 0);
     \node[] at(0, 0.005){$M_{\omega\bar 1}$}; \node[] at(1/2, 0.009){$I_{\omega}$}; \node[] at(1, 0.005){$M_{\omega\bar 0}$};

    %2th
     \draw[important line] (0, -0.02)--({2/5}, -0.02);  \draw[important line] ({3/5}, -0.02)--({1}, -0.02);
    \node[] at({2/5}, -0.025){\scriptsize{$M_{\omega1\bar 0}$}};  \node[] at({1/5}, -0.015){\scriptsize{$I_{\omega1}$}};
    \node[] at({3/5}, -0.025){\scriptsize{$M_{\omega0\bar 1}$}};  \node[] at({4/5}, -0.015){\scriptsize{$I_{\omega0}$}};
   % \node[] at({1/2}, -0.015){\scriptsize{$G_\omega$}};
    \node[] at(0, -0.025){\scriptsize{$M_{\omega\bar 1}$}}; \node[] at(1, -0.025){\scriptsize{$M_{\omega\bar 0}$}};
\end{tikzpicture}
}
\hspace{.2cm}
\subfigure[$M_{\omega1\bar 0}\geq M_{\omega0\bar 1}$]{
\begin{tikzpicture}[xscale=6,yscale=50,axis/.style={very thick, ->}, important line/.style={thick}, dashed line/.style={dashed, thin},
    pile/.style={thick, ->, >=stealth', shorten <=2pt, shorten>=2pt},every node/.style={color=black} ]

    %1th
    \draw[important line] (0, 0)--(1, 0);
     \node[] at(0, 0.005){$M_{\omega\bar 1}$}; \node[] at(1/2, 0.009){$I_{\omega}$}; \node[] at(1, 0.005){$M_{\omega\bar 0}$};

    %2th
     \draw[important line] (0, -0.02)--({3/5}, -0.02);  \draw[important line] ({2/5}, -0.021)--({1}, -0.021);
     %\draw[important line] ({3/5}, -0.019)--({3/5}, -0.022); \draw[important line] ({2/5}, -0.019)--({2/5}, -0.022);
    \node[] at({2/5}, -0.025){\scriptsize{$M_{\omega0\bar 1}$}};  \node[] at({1/5}, -0.015){\scriptsize{$I_{\omega1}$}};
    \node[] at({3/5}, -0.025){\scriptsize{$M_{\omega1\bar 0}$}};  \node[] at({4/5}, -0.015){\scriptsize{$I_{\omega0}$}};
    %\node[] at({1/2}, -0.015){\scriptsize{$O_\omega$}};
    \node[] at(0, -0.025){\scriptsize{$M_{\omega\bar 1}$}}; \node[] at(1, -0.025){\scriptsize{$M_{\omega\bar 0}$}};
\end{tikzpicture}
}
\caption{The sets $I_\omega$ with their endpoints and next level subintervals for an $\omega\in\{0,1\}^*$.}
\label{fi:2}
\end{figure}
\end{center}

The next lemma says that the lengths of the intervals $I_\omega$, and thus the values of $M_{\omega0\bar 1} - M_{\omega1\bar 0}$, only depend on the length $|\omega|$ and not on $\omega$ itself.

\begin{lemma}\label{le:a-length-eq}
Let $\alpha\in\mathcal{P}$ and $\kappa \ge 1$. Then for any $\omega, \eta \in \{ 0,1\}^\kappa$ we have $\lambda(I_\omega)= \lambda(I_\eta)$ and $M_{\omega0\bar 1} - M_{\omega1\bar 0} = M_{\eta0\bar 1} - M_{\eta1\bar 0}$.
\end{lemma}

\begin{proof}
By \eqref{q:interval},
\begin{align*}
\lambda(I_\omega)=M_{\omega\bar 0}-M_{\omega\bar 1}=\int_{[0,1]}F_{\omega\bar 1}-F_{\omega\bar 0}\, d\lambda.
\end{align*}
From \eqref{q:FepsilonminusF0} with $\varepsilon = \bar 1$ it follows that for any $\omega\in\{0,1\}^\kappa$, $\kappa \ge 1$, we have
\[ F_{\omega\bar 1}(z)-F_{\omega\bar 0}(z)= \sum_{n\in N_1^{\bar 1, \kappa}\setminus N_1^{\bar 0, \kappa}}(a_{\kappa+n}- t_{\kappa+n+1}z) +\sum_{n\in N_2^{\bar 1, \kappa}}a_{\kappa+n}z.\]
We see that $F_{\omega\bar 1}-F_{\omega\bar 0}$ only depends on $\kappa = |\omega|$ and $z$, so $\lambda(I_\omega)$ only depends on $\kappa$. Since $M_{\omega 0\bar 1} - M_{\omega 1\bar 0}= \lambda(I_\omega)-\lambda(I_{\omega 1})-\lambda(I_{\omega 0})$, we get the result.
\end{proof}
\vskip .2cm
Lemma \ref{le:a-length-eq} implies that for a given $\alpha\in\mathcal{P}$ we can define the function $G: \mathbb Z_{\ge 0} \to [0,1]$ by setting
\[\begin{split}
G(n):=M_{\omega 0\bar 1}-M_{\omega1\bar 0}=\ & \int_{[0,1]}F_{\omega1\bar 0}-F_{\omega0\bar 1}\, d\lambda\\
=\ & \int_{[0,1]}f_{n+1}^1-f_{n+1}^0d\lambda-\sum_{k\geq n+2}\int_{[0,1]}f_{k}^1-f_{k}^0\, d\lambda,
\end{split}\]
for any $\omega\in\{0,1\}^n$. The sign of $G$ will determine the structure of the set $\mathcal M$. The following proposition covers the case that the intervals $I_\omega$ eventually overlap. Write $I(n):= \lambda(I_\omega)$ for any $\omega \in \{0,1\}^n$.

\begin{proposition}\label{p:overlaps}
Let $\alpha \in \mathcal P$ and assume that there is an $N\geq 0$ such that $G(n) \le 0$ for all $n \ge N$. Then there are finitely many closed intervals $J_1, \ldots, J_K$, $K \in \mathbb N$, such that $\mathcal M = \bigcup_{k=1}^K J_k$. If, moreover, there exists an $N_1 \ge N$ such that $\frac{I(n-1)}{I(n)} \in (1, \frac{1+\sqrt 5}{2})$ for all $n > N_1$, then for all interior points $M$ of $\mathcal M$ there are uncountably many $\varepsilon \in \{0,1\}^\mathbb N$ such that $M=M_\varepsilon$.
\end{proposition}

\begin{proof}
Note that for each $n\geq 0$ the set $\bigcup_{\om\in\{0,1\}^n}I_\om$ is a union of at most $2^n$ closed intervals. The assumption that there is an $N\geq 0$ such that $G(n) \le 0$ for all $n \ge N$ implies that
$I_\omega = I_{\omega0}\cup I_{\omega1}$ for all $\omega \in \{0,1\}^*$ with $|\omega|\ge N$ and thus
\[ \mathcal M = \bigcap_{n=0}^\infty \bigcup_{\om\in\{0,1\}^n}I_\om = \bigcup_{\om\in\{0,1\}^N}I_\om.\]
This gives the first part of the statement.

\vskip .2cm
To obtain the second part, for each $n > N_1$ consider the maps
\[ \begin{split}
T_{n,1} : \ & \left[0,\frac{I(n)}{I(n-1)}\right]\to [0,1]; \, x \mapsto \frac{I(n-1)}{I(n)}x,\\
T_{n,0}: \ & \left[1-\frac{I(n)}{I(n-1)},1\right] \to [0,1]; \, x \mapsto \frac{I(n-1)}{I(n)}x - \frac{I(n-1)}{I(n)}+1.
\end{split}\]
For $x \in [0,1]$, let $\varepsilon = (\varepsilon_n)_{n \ge 1} \in \{0,1\}^\mathbb N$ be such that
\begin{equation}\label{q:norbit}
T^n_\varepsilon(x):=T_{N_1+n, \varepsilon_n} \circ \cdots \circ T_{N_1+1, \varepsilon_1} (x) \in [0,1]
\end{equation}
for all $n$. This means the following.
\begin{itemize}
\item If $T^{n-1}_\varepsilon (x) \leq 1-\frac{I(N_1+n)}{I(N_1+n-1)}$, then $\varepsilon_{n} =1$.
\item If $T^{n-1}_\varepsilon (x) \geq \frac{I(N_1+n)}{I(N_1+n-1)}$, then $\varepsilon_{n} =0$.
\item If $T^{n-1}_\varepsilon (x) \in \left[ 1-\frac{I(N_1+n)}{I(N_1+n-1)}, \frac{I(N_1+n)}{I(N_1+n-1)}\right]$, then $\varepsilon_{n}$ can be 0 or 1.
\end{itemize}
The condition that $\frac{I(n-1)}{I(n)} \in (1, \frac{1+\sqrt 5}{2})$ for all $n > N_1$ implies that
\[ T_{n,1}\left(1-\frac{I(n)}{I(n-1)}\right) = \frac{I(n-1)}{I(n)}-1 < \frac{\sqrt 5 -1}{2} < \frac{I(n+1)}{I(n)}\]
and by symmetry also $T_{n,0}\big(\frac{I(n)}{I(n-1)}\big) > 1- \frac{I(n+1)}{I(n)}$ for all $n > N_1$. Since $\frac{I(n-1)}{I(n)}>1$ for all $n$, this implies that if $T^{n-1}_\varepsilon (x) \not \in [ 1-\frac{I(n)}{I(n-1)}, \frac{I(n)}{I(n-1)}] \cup \{0,1\}$, then there is a $k> 1$ such that $T^{n+k-1}_\varepsilon (x) \in [ 1-\frac{I(n+k)}{I(n+k-1)}, \frac{I(n+k)}{I(n+k-1)}] $. From this we can deduce (basically as in \cite[Theorem 3]{Erdos-Joo-Komornik-1990}) that for each $x \in [0,1]$ there exist uncountably many sequences $\varepsilon = (\varepsilon_n)_{n \ge 1} \in \{0,1\}^\mathbb N$ for which \eqref{q:norbit} holds.

\vskip .2cm
Now, if we let $M \in I_\om$ for some $\om\in\{0,1\}^{N_1}$, then from the above we know that there are uncountably many sequences $\varepsilon \in \{0,1\}^\mathbb N$ such that $M =M_{\om \varepsilon}$. This gives the result.
\end{proof}

\vskip .2cm
Recall that a Cantor set in $\R$ is a nonempty compact set that has neither interior nor isolated points. A Cantor set is called homogeneous if it can be obtained in the following way. Let $(n_k)_{k \ge 1}$ be a sequence of positive integers and $(c_k)_{k \ge 1}$ a sequence of real numbers satisfying $n_k \ge 2$ and $0 < n_k c_k < 1$ for all $k \ge 1$. Recursively construct a sequence of intervals as follows. Start with an interval $I\subseteq \mathbb R$ and set $I_\epsilon = I$, called the level 0 interval. If for some $k \ge 0$ we have obtained the collection of $\prod_{j=1}^k n_j$ disjoint closed intervals of level $k$ (set $\prod_{j=1}^0 n_j=1$), then we obtain the collection of level $k+1$ intervals by  dividing each level $k$ interval $I_{i_1 \ldots i_k}$ into $n_{k+1}$ disjoint intervals of equal length $\lambda(I_{i_1 \ldots i_k})c_{k+1}$, labeled $I_{i_1 \ldots i_k 1}, \ldots, I_{i_1 \ldots i_k n_{k+1}}$ from left to right, in such a way that $I_{i_1 \ldots i_k 1}$ and $I_{i_1 \ldots i_k}$ share a common left endpoint, $I_{i_1 \ldots i_k n_{k+1}}$ and $I_{i_1 \ldots i_k}$ share a common right endpoint and distance between the right endpoint of $I_{i_1 \ldots i_k j}$ and the left endpoint of $I_{i_1  \ldots i_k (j+1)}$ is equal for all $1 \le j \le n_{k+1}-1$. The set
\[ \mathcal K =\bigcap_{k \ge 1} \bigcup_{i_1 \ldots i_k} I_{i_1 \ldots i_k}\]
is called a homogeneous Cantor set, see e.g.~\cite{Feng-Rao-wu-1997}. The Hausdorff dimension of a homogeneous Cantor set with sequences $(n_k)_{k \ge 1}$ and $(c_k)_{k \ge 1}$ is given by
\begin{equation}\label{q:homcantordim}
\dim_H (\mathcal K) = \liminf_{k \to \infty} \frac{\log (n_1 \cdots n_k)}{-\log(c_1 \cdots c_k)},
\end{equation}
see \cite[Lemma 2.2]{Feng-wen-wu-1997}, and the packing dimension is given by
\begin{equation}\label{q:homcantordimpacking}
\dim_P (\mathcal K) = \limsup_{k \to \infty} \frac{\log (n_1 \cdots n_{k+1})}{-\log(c_1 \cdots c_k)+\log n_{k+1}},
\end{equation}
see \cite[Theorem 3.1]{Feng-wen-wu-1997}. We have the following result.

\begin{proposition}\label{p:gaps}
Let $\alpha \in \mathcal P$. The following statements hold.
\begin{itemize}
%\item[(i)] If there is an $N \in \mathbb N$ such that $G(n) \le 0$ for all $n \ge N$, then there are finitely many closed intervals $J_1, \ldots, J_K$, $K \in \mathbb N$, such that $\mathcal M = \bigcup_{k=1}^K J_k$.
\item[(i)] If $G(n)>0$ for all $n\geq 0$, then $\mathcal M$ is a homogeneous Cantor set.
\item[(ii)] If there is an $N \in \mathbb N$ such that $G(n) > 0$ for all $n \ge N$, then $\mathcal M$ is a Cantor set.
\end{itemize}
In both cases
\begin{equation}\label{q:dimM}
\dim_H(\mathcal M) = \liminf_{k \to \infty} \frac{k\log 2}{- \log (I(k))}\quad and\quad
\dim_P(\mathcal M) = \limsup_{k \to \infty} \frac{k\log 2}{- \log (I(k))}.
\end{equation}
\end{proposition}

\begin{proof}
Part (i) of the proposition follows from the structure of the sets $I_\omega$ as described in the proof of Theorem~\ref{thm:main-1}, from \eqref{eq:structure-M-se-1} and Lemma~\ref{le:a-length-eq} with the sequences $(n_k)_{k \ge 1}$ and $(c_k)_{k \ge 1}$ given by $n_k =2$ and $c_k = \frac{I(k)}{I(k-1)}$ for all $k\ge 1$.
By \eqref{q:homcantordim}, we have
\begin{align*}
    \dim_H (\mathcal{M})=\liminf_{k\to\infty}\frac{-k\log 2}{\log \left(\frac{I(1)}{I(0)}\ldots \frac{I(k)}{I(k-1)}\right) }=\liminf_{k\to\infty}\frac{k\log 2}{\log (I(0))-\log (I(k))} =\liminf_{k\to\infty}\frac{k\log 2}{-\log (I(k))},
    \end{align*}
where the last equation follows since $I(0)$ is a constant and $\lim_{k\to\infty}I(k)\to 0$. The packing dimension follows similarly from
\eqref{q:homcantordimpacking}.
For (ii), note that for any word $\eta \in \{0,1\}^N$ the set
\[ \bigcap_{n=0}^\infty \bigcup_{\om\in\{0,1\}^n}I_{\eta\om} \]
is a homogeneous Cantor set.
Since
\[ \mathcal M = \bigcup_{\eta \in \{0,1\}^N} \bigcap_{n=0}^\infty \bigcup_{\om\in\{0,1\}^n}I_{\eta\om} \]
and a finite union of Cantor sets is again a Cantor set, we also get (ii). For the Hausdorff dimension we get
\begin{align*}
    \dim_H (\mathcal{M})=&\dim_H \left(\bigcup_{\eta \in \{0,1\}^N} \bigcap_{n=0}^\infty \bigcup_{\om\in\{0,1\}^n}I_{\eta\om}\right) \\
    =&\max_{\eta\in\{0,1\}^N} \dim_H \left(\bigcap_{n=0}^\infty \bigcup_{\om\in\{0,1\}^n}I_{\eta\om}\right)=\liminf_{k\to\infty}\frac{k\log 2}{-\log (I(k))}. %\qedhere
    \end{align*}
The packing dimension is given by the same argument. \end{proof}
\vskip .2cm

Proposition~\ref{p:overlaps} and Proposition~\ref{p:gaps} show that to determine the structure of $\mathcal M$ it is useful to consider the function $G$ in more detail. First we consider the indices $k$ such that $a_k/t_{k+1}\geq 1$ which is equivalent to $0< t_{k+1}/t_k\leq 1/2 $. By \eqref{eq:F-sum-f} we have
\begin{align*}
\begin{split}
\int_{[0,1]}f_k^1-f_k^0\, d\lambda =\ &\int_{[0,\frac{a_k}{t_k}]}(t_k-t_{k+1})z\, d\lambda(z) +\int_{[\frac{a_k}{t_k},1]}a_k-t_{k+1}z \, d\lambda(z)\\
%=&\frac{(t_n-t_{n+1})a_n^2}{2t_n^2}+a_n-\frac{a_n^2}{t_n}-\frac{t_{n+1}}{2}\left(1-\frac{a_n^2}{t_{n}^2}\right)\\
=\ &\frac{a_k^2}{2t_k}-\frac{t_{k+1}a_k^2}{2t_k^2}+a_k-\frac{a_k^2}{t_k}-\frac{t_{k+1}}{2}+\frac{t_{k+1}a_k^2}{2t_k^2}\\
=\ &a_k-\frac{a_k^2}{2t_k}-\frac{t_{k+1}}{2}=\frac{t_k}{2}-\frac{t_{k+1}}{2}-\frac{t_{k+1}^2}{2t_k}.
\end{split}
\end{align*}
Next we consider the indices $k$ such that $a_k/t_{k+1}< 1$ which is equivalent to $1/2< t_{k+1}/t_k< 1 $. By \eqref{eq:F-sum-f} we have
\begin{align*}
\begin{split}
\int_{[0,1]}f_k^1-f_k^0d\lambda=&\int_{[0,\frac{a_k}{t_k}]}(t_k-t_{k+1})z \, d\lambda(z) +\int_{[\frac{a_k}{t_k},\frac{a_k}{t_{k+1}}]}a_k-t_{k+1}z\, d\lambda(z)\\
%=&\int_{[\frac{a_n}{t_n},\frac{a_n}{t_{n+1}}]}a_n-t_{n+1}zd\lambda+
%\int_{[\frac{a_n}{t_{n+1}},1]}a_nzd\lambda\\
=&\frac{a_k^2}{2t_k}-\frac{t_{k+1}a_k^2}{2t_k^2}+\frac{a_k^2}{t_{k+1}}-\frac{a_k^2}{t_k}-\frac{a_k^2}{2t_{k+1}}+\frac{t_{k+1}a_k^2}{2t_k^2}\\
=&\frac{a_k^2}{2t_{k+1}}-\frac{a_k^2}{2t_k}=\frac{t_k^2}{2t_{k+1}}-\frac{t_{k+1}^2}{2t_{k}}+\frac{3t_{k+1}}{2}-\frac{3t_{k}}{2}.
%=\frac{a_k^4}{2t_n^2t_{n+1}}+\frac{a_n}{2}-\frac{a_n^3}{2t_{n+1}^2}.
\end{split}
\end{align*}
For any $k \ge 1$ set
\begin{equation}\label{q:smallg}
g(k):=\int_{[0,1]}f_k^1-f_k^0\, d\lambda=
\begin{cases}
\frac{t_k}{2}-\frac{t_{k+1}}{2}-\frac{t_{k+1}^2}{2t_{k}} \quad & \text{if } 0\leq \frac{t_{k+1}}{t_k}\leq \frac{1}{2};\\
\\
\frac{t_k^2}{2t_{k+1}}-\frac{t_{k+1}^2}{2t_{k}}+\frac{3t_{k+1}}{2}-\frac{3t_{k}}{2} \quad & \text{if } \frac{1}{2}<\frac{t_{k+1}}{t_k}\leq 1.\\
\end{cases}
\end{equation}
So, for any given $\alpha\in\mathcal{P}$ we see that $g(k)$ only depends on $k$ and for any $\omega\in\{0,1\}^n$ we have
\begin{equation}\label{eq:def-G}
G(n)=g(n+1)-\sum_{k=n+2}^\infty g(k).
\end{equation}
Then in the $(n+1)$-level of the construction of $\mathcal M_\alpha$ we see gaps as in Figure~\ref{fi:2}(a) if $G(n)<0$ and overlaps as in Figure~\ref{fi:2}(b) if $G(n)\ge 0$. We use this to describe the structure of $\mathcal M_\alpha$ further under some additional restrictions below. For each $n \ge 1$ set
\[ \rho_n := \frac{t_{n+1}}{t_n}.\]

\subsection{$\alpha\in\mathcal{P}$ with $1/2 < \rho_n < 1$ for all large enough $n$.}
In this section we assume that there is an $N \ge 0$ such that $\rho_{n+1} > \frac12$ for all $n \ge N$. By \eqref{q:smallg} and \eqref{eq:def-G} we then have for all $n \ge N$ that
\begin{equation}\label{eq:G-N-p-ge-1-2}
\begin{split}
G(n)=\ &\frac{t_{n+1}^2}{2t_{n+2}}-\frac{t_{n+2}^2}{2t_{n+1}}+\frac{3t_{n+2}}{2}-\frac{3t_{n+1}}{2}
 -\sum_{k=n+2}^\infty \left(\frac{t_k^2}{2t_{k+1}}-\frac{t_{k+1}^2}{2t_{k}}+\frac{3t_{k+1}}{2}-\frac{3t_{k}}{2}\right)\\
=\ &\frac{t_{n+1}}{2\rho_{n+1}}-\frac{\rho_{n+1}t_{n+2}}{2}+{3t_{n+2}}-\frac{3t_{n+1}}{2}
 -\sum_{k=n+2}^\infty \frac{t_{k}}{2\rho_{k}}+\sum_{k=n+2}^\infty \frac{\rho_{k}t_{k+1}}{2}\\
=\ &\left(\frac{1}{2\rho_{n+1}}-\frac{\rho_{n+1}^2}{2}+{3\rho_{n+1}}-\frac{3}{2}\right)t_{n+1}
 -\sum_{k=n+2}^\infty \left(\frac{1}{2\rho_{k}}-\frac{\rho_{k}^2}{2}\right)t_{k}\\
  =&\left(\frac{1}{2\rho_{n+1}}-\frac{\rho_{n+1}^2}{2}+{3\rho_{n+1}}-\frac{3}{2}
 -\sum_{k=n+2}^\infty \left(\frac1{2\rho_{k}}-\frac{\rho_{k}^2}{2}\right)\prod_{i=n+1}^{k-1}\rho_i \right)t_{n+1}.
\end{split}
\end{equation}
\vskip .2cm
We first prove the following technical lemma. Set $s_k:=\sup_{n>k}\rho_n$, $k  \geq 0$.

\begin{lemma}\label{le:3-2}
Let $\alpha\in\mathcal{P}$. Assume that there is an $N \ge 0$ such that $1/2< \rho_{n+1}< 1$ for all $n \ge N$ and that there is an $N_1 \ge N$ such that $s_{n+1} \leq 1-\frac{7\rho_{n+1}^2}{8\rho_{n+1}^3+4}$ for all $n\geq N_1$. Then $\mathcal{M}=\bigcup_{k=1}^{K}J_k$ for finitely many closed intervals $J_1,...,J_K$, $K\in\mathbb{N}$.
\end{lemma}

\begin{proof}
Fix some $n \ge N_1$ and set
$$h_{n+1}:=\sum_{k=n+2}^\infty
\left(\frac{1}{\rho_{k}}-\rho_{k}^2\right) \prod_{i=n+1}^{k-1}\rho_i.$$
Fix a $k\geq n+2$. Then the derivative of $h_{n+1}$ with respect to $\rho_k$ is given by
\begin{equation*}
\frac{\partial h_{n+1}}{\partial \rho_k}=\prod_{i=n+1}^{k-1}\rho_i\left(-\frac{1}{\rho_{k}^2}-2\rho_{k}+
\sum_{j=k+1}^\infty  \prod_{i=k+1}^{j-1}\rho_i\left(\frac{1}{\rho_j}-\rho_j^2\right)\right),
\end{equation*}
where $\prod_{i=k+1}^k\rho_i=1$. Since the map $x \mapsto \frac1x -x^2$ is decreasing on the interval $(\frac12,1]$, we have
\begin{align*}
%%\begin{split}
\frac{\partial h_{n+1}}{\partial \rho_k}< &\prod_{i=n+1}^{k-1}\rho_i\left(-\frac{1}{\rho_{k}^2}-2\rho_{k}+
\frac{7}{4}\sum_{j=k+1}^\infty \left( \prod_{i=k+1}^{j-1}\rho_i\right)\right)\\
\leq & \prod_{i=n+1}^{k-1}\rho_i\left(-\frac{1}{\rho_{k}^2}-2\rho_{k}+
\frac{7}{4\left(1-s_{k}\right)}\right).
%%\begin{split}
\end{align*}
Since $n \ge N_1$ this implies that $\frac{\partial h_{n+1}}{\partial \rho_k}< 0$, hence $h_{n+1}$ is decreasing as a function of $\rho_k$. Since we took $k$ arbitrary, we obtain that $\frac{\partial h_{n+1}}{\partial \rho_k} < 0$ for all $k \ge n+2$. Moreover for all $k \ge n+2$ we have $\rho_k > \frac12$. This gives
\begin{equation*}
\begin{split}
G(n)=\ & \left(\frac{1}{2\rho_{n+1}}-\frac{\rho_{n+1}^2}{2}+{3\rho_{n+1}}-\frac{3}{2} -\sum_{k=n+2}^\infty \left(\frac1{2\rho_{k}}-\frac{\rho_{k}^2}{2}\right)\prod_{i=n+1}^{k-1}\rho_i \right)t_{n+1}\\
 \leq \ & \left(\frac{1}{2\rho_{n+1}}-\frac{\rho_{n+1}^2}{2}+{3\rho_{n+1}}-\frac{3}{2}- \left(\frac{1}{2s_{n+1}}+\frac{1}{2}+\frac{s_{n+1}}{2}\right)\rho_{n+1} \right)t_{n+1} \leq 0,
\end{split}
\end{equation*}
where the last inequality follows since $\frac{1}{2\rho_{n+1}^2}-\frac{\rho_{n+1}}{2}+{3}-\frac{3}{2\rho_{n+1}}\leq \frac{1}{2s_{n+1}}+\frac{1}{2}+\frac{s_{n+1}}{2}$ for all $\rho_{n+1},s_{n+1}\in[1/2,1]$. The result follows from Proposition~\ref{p:overlaps}.
\end{proof}
\vskip .2cm
In case $(\rho_n)_{n \ge 1}$ converges to a limit $\rho \neq 1$, we have the following result.

\begin{proposition}\label{prop:3-2}
Let $\alpha \in \mathcal P$. Assume that there is an $N \ge 0$ such that $\rho_{n+1} > \frac12$ for all $n \ge N$. Assume furthermore that $\rho = \lim_{n \to \infty} \rho_n$ exists and $\frac12 \le \rho < 1$. Then there are finitely many closed intervals $J_1, \ldots, J_K$, $K \in \mathbb N$, such that $\mathcal M_\alpha = \bigcup_{k=1}^K J_k$.
\end{proposition}

\begin{proof}
To obtain the result, by \eqref{eq:G-N-p-ge-1-2} it is enough to show that there exists a $N_\delta \ge N$ such that
\[  \frac{1}{2\rho_{n+1}}-\frac{\rho_{n+1}^2}{2}+{3\rho_{n+1}}-\frac{3}{2}
 \le \sum_{k=n+2}^\infty \left(\frac1{2\rho_{k}}-\frac{\rho_{k}^2}{2}\right)\prod_{i=n+1}^{k-1}\rho_i, \quad\text{for all } n\geq N_1.\]

\vskip .2cm
First consider $\frac 12 < \rho < 1$. For $\delta >0$ let $N_\delta$ be such that $\max\{\frac12, \rho-\delta\} <\rho_n<\rho+\delta$ for all $n\geq N_\delta$. Then for all $n\geq N_\delta$
 \begin{align*}
\frac{1}{2\rho_{n+1}}&-\frac{\rho_{n+1}^2}{2}+{3\rho_{n+1}}-\frac{3}{2} -\sum_{k=n+2}^\infty \left(\frac1{2\rho_{k}}-\frac{\rho_{k}^2}{2}\right)\prod_{i=n+1}^{k-1}\rho_i \\
<\ &\frac{1}{2(\rho+\delta)}-\frac{(\rho+\delta)^2}{2}+{3(\rho+\delta)}-\frac{3}{2} -\sum_{k=n+2}^\infty \left(\frac1{2(\rho+\delta)}-\frac{(\rho+\delta)^2}{2}\right)(\rho-\delta)^{k-n-1}\\
=\ &\frac{1}{2(\rho+\delta)}-\frac{(\rho+\delta)^2}{2}+{3(\rho+\delta)}-\frac{3}{2} - \left(\frac1{2(\rho+\delta)}-\frac{(\rho+\delta)^2}{2}\right)\frac{\rho-\delta}{1-(\rho-\delta)}\\
%=\ &\left(\frac1{2(\rho+\delta)}-\frac{(\rho+\delta)^2}{2}\right)\frac{1-2\rho+2\delta}{1-(\rho-\delta)}+{3(\rho+\delta)}-\frac{3}{2}\\
 =\ & \frac{2\rho^4-7\rho^3+9\rho^2-5\rho+1}{2(\rho+\delta)(1-\rho+\delta)}+\delta\cdot\frac{4\rho^3-9\rho^2+12\rho-1-4\rho\delta^2+3\rho\delta-2\delta^3+5\delta^2+3\delta}{2(\rho+\delta)(1-\rho+\delta)}
 \end{align*}
 where the first inequality follows since $x \mapsto \frac{1}{2x}-\frac{x^2}{2}+{3x}-\frac{3}{2}$ is increasing on the interval $[\frac12,1]$ and $x \mapsto \frac1{2x}-\frac{x^2}{2}$ is decreasing on the interval $[\frac12,1]$. Note that $\rho \in (\frac12,1)$ implies that $2\rho^4-7\rho^3+9\rho^2-5\rho+1 < 0$. Hence, we can choose a $\delta >0$ such that
 $$\frac{1}{2\rho_{n+1}}-\frac{\rho_{n+1}^2}{2}+{3\rho_{n+1}}-\frac{3}{2} -\sum_{k=n+2}^\infty \left(\frac1{2\rho_{k}}-\frac{\rho_{k}^2}{2}\right)\prod_{i=n+1}^{k-1}\rho_i\leq0,$$
 and thus $G(n)\leq 0$  for all $n\geq  N_\delta$. Proposition~\ref{p:overlaps} then gives the result.

 \vskip .2cm
Next we consider $\rho=1/2$. This implies the existence of an $ N_1\in\mathbb{N}$ such that $\frac12 < \rho_{n+1}< 0.55$ for all $n\geq N_1$. For such $n$ we have $s_{n+1}\le 0.55$ and $1-\frac{7\rho_{n+1}^2}{8\rho_{n+1}^3+4} \ge 1-\frac{7\cdot 0.55^2}{8\cdot 0.55^3+4}=0.6027...$, so the result follows from Lemma \ref{le:3-2}
\end{proof}

\subsection{$\alpha\in\mathcal{P}$ with $0<\rho_n \leq 1/2$ for all large enough $n$}
Let $\alpha \in \mathcal P$ and assume that there exists an $N \geq 0$ such that $0<\rho_{n+1} \leq 1/2$ for all $n \ge N$. By (\ref{q:smallg}) and (\ref{eq:def-G}) we then have for $n \ge N$ that
\begin{equation}\label{eq:G-N-p-le-1-2}
\begin{split}
G(n)=& \frac{t_{n+1}}{2}-\frac{t_{n+2}}{2}-\frac{t_{n+2}^2}{2t_{n+1}}-\sum_{k=n+2}^\infty\left(\frac{t_k}{2}-\frac{t_{k+1}}{2}-\frac{t_{k+1}^2}{2t_k}\right)\\
=&\frac{t_{n+1}}{2}-{t_{n+2}}-\frac{t_{n+2}^2}{2t_{n+1}}+\sum_{k=n+2}^\infty\frac{t_{k+1}^2}{2t_k}\\
=&\left(\frac{1}{2}-\rho_{n+1}-\frac{\rho_{n+1}^2}{2}\right)t_{n+1}+\sum_{k=n+2}^\infty\frac{t_{k+1}^2}{2t_k}\\
%=&\left(\frac{1}{2}-\rho_{n+1}-\frac{\rho_{n+1}^2}{2}\right)t_{n+1}+\left(\sum_{k=n+2}^\infty\frac{\rho_k\prod_{i=n+1}^k \rho_i}{2}\right)t_{n+1}\\
=&\left(\frac{1}{2}-\rho_{n+1}-\frac{\rho_{n+1}^2}{2}+\sum_{k=n+2}^\infty\frac{\rho_k}{2}\prod_{i=n+1}^k \rho_i\right)t_{n+1}.
\end{split}
\end{equation}
%Then
%\begin{align}\label{eq:G-n-le-ge-0}
%G(n)>0 &\quad \text{if}\quad \frac{1}{2}-\rho_{n+1}-\frac{\rho_{n+1}^2}{2}+\sum_{k=n+2}^\infty\frac{\rho_k\prod_{i=n+1}^k \rho_i}{2}>0\\
%G(n)\leq 0 &\quad \text{if}\quad \frac{1}{2}-\rho_{n+1}-\frac{\rho_{n+1}^2}{2}+\sum_{k=n+2}^\infty\frac{\rho_k\prod_{i=n+1}^k \rho_i}{2}\leq 0.
%\end{align}
Before we give results on the structure of $\mathcal M$, we first give the following lemma. Recall that we have set $s_n=\sup_{k> n}\rho_k$. Also set $m_n:=\inf_{k> n}\rho_k$.

\begin{lemma}\label{prop:111}
Let $\alpha \in \mathcal P$ be a partition for which there exists an $N \geq 0$ such that $0<\rho_{n+1}\leq 1/2$ for all $n \ge N$. We have the following two statements.
\begin{itemize}
\item[(i)] If $\rho_{n+1}< \sqrt{2+\frac{m_{n}^3}{1-m_{n}}}-1$, then $G(n)>0$.
\item[(ii)] If $\rho_{n+1}\geq \sqrt{2+\frac{s_{n}^3}{1-s_{n}}}-1$, then $G(n)\leq0$.
\end{itemize}
\end{lemma}

\begin{proof}
 By (\ref{eq:G-N-p-le-1-2}) we have
 \begin{align*}
 G(n)  \geq&\left(\frac{1}{2}-\rho_{n+1}-\frac{\rho_{n+1}^2}{2}+\sum_{k=3}^\infty\frac{m_{n}^k}{2}\right)t_{n+1}\\
 =&\left(\frac{1}{2}-\rho_{n+1}-\frac{\rho_{n+1}^2}{2}+\frac{m_{n}^3}{2(1-m_{n})}\right)t_{n+1}.
 \end{align*}
Hence, $G(n)>0$ if $\rho_{n+1}< \sqrt{2+\frac{m_{n}^3}{1-m_{n}}}-1$. Similarly, we have
\[ G(n) \leq \left(\frac{1}{2}-\rho_{n+1}-\frac{\rho_{n+1}^2}{2}+\frac{s_{n}^3}{2(1-s_{n})}\right)t_{n+1}.\]
So, $G(n)\leq 0$ if $\rho_{n+1}\geq \sqrt{2+\frac{s_{n}^3}{1-s_{n}}}-1$.
\end{proof}
\vskip .2cm

\begin{proposition}\label{prop:G-n-le-sq-2-1}
Let $\alpha\in\mathcal{P}$. The following statements hold.
\begin{itemize}
\item[(i)] If there is an $N \geq 0$ such that $0< \rho_{n+1}\leq \sqrt{2}-1$ for all $n \ge N$, then $\mathcal{M}$ is a Cantor set. In particular, if $0< \rho_{n+1}\leq \sqrt{2}-1$ for all $n\geq 0$, then $\mathcal{M}$ is a homogeneous Cantor set.% and the Hausdorff dimension of $\mathcal M$ is given by \eqref{q:dimM}.
\item[(ii)] If $\rho= \lim_{n\to\infty}\rho_n$ exists and $\rho \in (0, \frac12)$, then $\mathcal{M}$ is a Cantor set.% and the Hausdorff dimension of $\mathcal M$ is also given by \eqref{q:dimM}.
\item[(iii)] Assume that $\rho= \lim_{n\to\infty}\rho_n$ exists and $\rho = \frac12$.
\begin{itemize}
\item[(a)] If there is an $N \in \mathbb N$ such that $(\rho_n)_{n \ge N}$ is increasing and $\rho_n \neq \frac12$ for any $n\ge N$, then $\mathcal{M}$ is a Cantor set. In particular, if $(\rho_n)_{n \ge 1}$ is an increasing sequence with $\rho_n \neq \frac12$ for any $n$, then $\mathcal{M}$ is a homogeneous Cantor set.% and the Hausdorff dimension of $\mathcal M$ is again given by \eqref{q:dimM}.
%\item[(b)] If there is an $N \in \mathbb N$ such that $(\rho_n)_{n \ge N}$ is increasing and $\rho_n \neq \frac12$ for any $n\ge N$, then $\mathcal{M}$ is a Cantor set.
\item[(b)] If there is an $N \in \mathbb N$ such that $\rho_n =\frac12$ for all $n \ge N$, then $\mathcal M$ is a finite union of closed intervals.
\end{itemize}
\end{itemize}
In all the above cases where $\mathcal{M}$ is a Cantor set, the Hausdorff dimension and packing dimension of $\mathcal M$ are given by \eqref{q:dimM}.
\end{proposition}

\begin{proof}
For (i), for each $n\geq 0$ such that $0<\rho_{n+1}\leq \sqrt{2}-1$ we have $\frac{1}{2}-\rho_{n+1}-\frac{\rho_{n+1}^2}{2}\geq 0$. From \eqref{eq:G-N-p-le-1-2} we see that this gives
$$G(n)=\left(\frac{1}{2}-\rho_{n+1}-\frac{\rho_{n+1}^2}{2}\right)t_{n+1}+\sum_{k=n+2}^\infty\frac{t_{k+1}^2}{2t_k}>0,$$
which implies that $[M_{\om 1^\infty},M_{\om 10^\infty}] \cap [M_{\om 0 1^\infty},M_{\om 0^\infty}] = \emptyset$ for each $\omega\in\{0,1\}^n$. If this holds for all $n\geq 0$, we get a homogeneous Cantor set by Proposition~\ref{p:gaps}(i). If this is only true for all $n$ large enough, then we get a Cantor set by Proposition~\ref{p:gaps}(ii). The fractal dimensions are then given by \eqref{q:dimM}.

\vskip .2cm
For (ii), by the assumption it holds that $m_{n} \le \rho < \frac12$ for all $n \ge 0$ and that $\lim_{n \to \infty} m_{n} = \rho$. The map $f:x \mapsto \sqrt{2+\frac{x^3}{1-x}}-1-x$ is strictly decreasing on the interval $[0,\frac12]$ with $f(\rho)>0$. Hence, there is an $N_1 \in \mathbb N$ such that
\[ \rho_{n+1} - m_{n} <  f(\rho) \le f(m_{n}) \]
for all $n \ge N_1$. This implies that $\rho_{n+1} < \sqrt{2+\frac{m_{n}^3}{1-m_{n}}}-1$ for all $n \ge N_1$. Lemma~\ref{prop:111}(i) and Proposition~\ref{p:gaps}(ii) combined then give the result.

\vskip .2cm
We now turn to (iii). First we consider $(\rho_n)$ is increasing. By assumption there is an $N_1 \in \mathbb N$ such that $\rho_{n+1} \le \sqrt 2 -1$ for all $n < N_1$ and $\rho_{n+1} > \sqrt 2 -1$ for all $n \ge N_1$. The proof of (i) shows that $G(n)>0$ for all $n < N_1$. So we focus on $n \ge N_1$. Since $\rho_{n+1}=\frac{t_{n+2}}{t_{n+1}}\leq \frac{1}{2}$, we have $a_{n+1}=t_{n+1}-t_{n+2}\geq t_{n+2}$ and thus $t_{n+1}=\sum_{k=n+1}^\infty a_k\geq \sum_{k=n+2}^\infty t_k$. By \eqref{eq:G-N-p-le-1-2} we then get
 \begin{align*}
 G(n) \geq & \left(\frac{1}{2}-\rho_{n+1}-\frac{\rho_{n+1}^2}{2}\right)\sum_{k=n+2} t_k +\sum_{k=n+2}^\infty\frac{\rho_k^2t_{k}}{2}\\
 =& \sum_{k=n+2}\left(\frac{1}{2}-\rho_{n+1}-\frac{\rho_{n+1}^2}{2}+\frac{\rho_k^2}{2}\right) t_k.
 \end{align*}
If $(\rho_n)_{n \ge 1}$ is an increasing sequence with $\rho_n \neq \rho_{n+1}$ infinitely often, we see that $G(n)>0$ for all $n\geq 0$, which makes $\mathcal M$ a homogeneous Cantor set. The Hausdorff dimension and packing dimension of $\mathcal M$ are given by \eqref{q:dimM}. If $(\rho_n)_{n \ge N}$ is an increasing sequence we see from the previous part that $G(n)>0$ for all $n \ge N$ and hence $\mathcal M$ is a Cantor set by Proposition~\ref{p:gaps}(ii). Finally, if there is an $N \ge 0$ such that $\rho_{n+1} =\frac12$ for all $n \ge N$, then $G(n)=0$ for all $n \ge N$. This implies that for each $n \ge N$ we get
\[ \bigcup_{\omega \in \{0,1\}^n} I_\omega = \bigcup_{\omega \in \{0,1\}^{N}} I_\omega.\]
This gives the result.

\end{proof}

As an example of the above result, let us consider a partition that frequently occurs in relation to generalised L\"uroth transformations, namely $\alpha=\left\{\left(\frac{1}{2^{n}},\frac{1}{2^{n-1}}\right]:n\in\mathbb{N}\right\} \cup \{0\}$, see e.g.~\cite{Dajani-deLepper-Robinson-2020}.

\begin{example}
Let $\alpha=\left\{\left(\frac{1}{2^{n}},\frac{1}{2^{n-1}}\right]:n\in\mathbb{N}\right\}\cup \{0\}$, so that $a_n=\frac1{2^n}$ and $t_n=\frac1{2^{n-1}}$ and hence $\rho_n=\frac{1}{2}$ for all $n\in\mathbb{N}$. This means that we are in the setting of Proposition~\ref{prop:G-n-le-sq-2-1}(iii). In fact we immediately see that $G(n)=0$ for all $n \geq 0$, so $\mathcal M$ is an interval. We compute the boundary points of $\mathcal M$. For any $\ep\in\{0,1\}^\mathbb{N}$,
\begin{align}\label{eq:F-a-1-2-n}
F_\ep(z)=\sum_{n\, :\,  z \le 1/2^{\ep_n}}\frac{z}{2^{n-\ep_n}}+\sum_{n\, :\,  z> 1/2^{\ep_n}}\frac{1}{2^n}.
\end{align}
Since $z \le 1$ we obtain
\[ M_{\bar 0}=\int_{[0,1]}(1-F_{\bar 0})d\lambda=\int_{[0,1]}\left(1-\sum_{n=1}^{\infty}\frac{z}{2^n}\right)d\lambda(z)=1-\int_{[0,1]}z d\lambda(z)=\frac{1}{2}.\]
On the other hand,
\[\begin{split}
F_{\bar 1}(z)=\left\{
\begin{array}{ll}
\sum_{n= 1}^\infty\frac{z}{2^{n-1}} \quad & \text{if } z\leq \frac{1}{2};\\
\sum_{n= 1}^\infty \frac{1}{2^{n}} \quad & \text{if } z>\frac{1}{2},\\
\end{array}
\right.
\end{split}\]
so that
\[ M_{\bar 1}=\int_{[0,\frac{1}{2}]}\left(1-\sum_{n = 1}^\infty\frac{z}{2^{n-1}}\right)d\lambda(z) +\int_{(\frac{1}{2},1]}\left(1-\sum_{n= 1}^\infty\frac{1}{2^n}\right)d\lambda(z) = \int_{[0,\frac{1}{2}]}\left(1-2z\right)d\lambda(z) = \frac14.\]
Hence,
$$\mathcal{M}=[M_{\bar 1},M_{\bar 0}]=\left[1/4,1/2\right].$$
Let $M\in \mathcal M$ be such that it is not an endpoint of an interval $I_\omega$ for any $\omega \in \{0,1\}^n,\ n\geq 1$. Then there is a unique $\varepsilon$ such that $M=M_\varepsilon$. For the countably many points $M\in (\frac14, \frac12)$ that are an endpoint of an interval $I_\omega$ there are precisely two sequences $\varepsilon$ with $M=M_\varepsilon$, one ending in $\bar 0$ and one ending in $\bar 1$.
\end{example}

The content of Theorem~\ref{thm:main-2} is covered by the combined results of Proposition~\ref{prop:3-2} and Proposition~\ref{prop:G-n-le-sq-2-1}.

\section{The L\"uroth partition and other examples}\label{sec:4}
The results from the previous section do not describe the set $\mathcal M$ for all possible partitions $\alpha \in \mathcal P$ for which $\rho = \lim_{n \to \infty} \rho_n$ exists. For $\rho=\frac12$, what is not covered is the case that $\rho=\frac12$, there are infinitely many $n$ such that $\rho_n < \frac12$ and $(\rho_n)_{n \ge 1}$ is not increasing. Another case that is not addressed is when $\lim_{n \to \infty} \rho_n$ exists and equals 1. Since in this case
\[ \lim_{n \to \infty} \left(\frac1{2\rho_{n+1}} - \frac{\rho_{n+1}^2}{2} + 3\rho_{n+1}\right)-\frac32 = \frac32,\]
we see from \eqref{eq:G-N-p-ge-1-2} that the structure of $\mathcal M$ depends on the limiting behaviour of
$$L_k:=\frac{\left(\frac1{2\rho_{k+1}}-\frac{\rho_{k+1}^2}{2}\right)\prod_{i=n+1}^{k}\rho_i}{\left(\frac1{2\rho_{k}}-\frac{\rho_{k}^2}{2}\right)\prod_{i=n+1}^{k-1}\rho_i}=\frac{\rho_k^2(1-\rho_{k+1}^3)}{\rho_{k+1}(1-\rho_k^3)}, \quad k\in\mathbb{N}.$$
Set $R:=\limsup L_k$ and $r:=\liminf L_k$. Then $\sum_{k=1}^\infty \left(\frac1{2\rho_{k}}-\frac{\rho_{k}^2}{2}\right)\prod_{i=n+1}^{k-1}\rho_i$ is divergent if $r>1$. This would imply that $G(n)\leq 0$ for all $n$ large enough, so that $\mathcal M$ would become a finite union of closed intervals by Proposition~\ref{p:overlaps}. %{p:gapsoverlaps}(i).
On the other hand, $\sum_{k=1}^\infty \left(\frac1{2\rho_{k}}-\frac{\rho_{k}^2}{2}\right)\prod_{i=n+1}^{k-1}\rho_i$ is convergent if $R<1$, and in which case for all $n$ large enough $G(n) >0$ for all $n$ large enough. Proposition~\ref{p:gaps}(i) implies that $\mathcal M$ would be a Cantor set in this case. Finally, if $r=R=1$ we cannot give a general result on the structure of $\mathcal M$. Note that the L\"uroth partition $\alpha_L = \{ ( \frac1{n+1}, \frac1{n}] \}_{n \ge 1}\cup \{0\}$ falls in this last category as then $\rho_n = \frac{n}{n+1}$ for all $n \ge 1$ and thus
\[ L_k = \frac{(\frac{k}{k+1})^2(1-(\frac{k+1}{k+2})^3)}{\frac{k+1}{k+2} (1-(\frac{k}{k+1})^3)} = \frac{k^2((k+2)^3-(k+1)^3)}{(k+1)^2 ((k+1)^3-k^3)} \to 1
\]
as $k \to \infty$. Nonetheless, we have the following result for $\alpha_L$.
\begin{proposition}\label{p:luroth8}
$\mathcal M_{\alpha_L}$ is the finite union of eight closed intervals.
\end{proposition}

\begin{proof}
For $\alpha_L$ we have $\rho_1 = \frac12$ and $\rho_n = \frac{n}{n+1} > \frac12$ for all $n >1$. Computing the function $g$ yields $g(1)= \frac18$ and
\[ g(n) = \frac{n+1}{2n^2} - \frac{n}{2(n+1)^2} - \frac32 \frac1{n(n+1)} = \frac1{2n^2(n+1)^2}\]
for $n \ge 2$. Using the fact that $\sum_{k \ge 1} \frac1{2k^2(k+1)^2} = \frac{\pi^2-9}{6}$, we obtain
\[ G(0) = g(1) - \sum_{k \ge 2} \frac1{2k^2(k+1)^2} = \frac18 + \frac{9-\pi^2}{6} + \frac18 = \frac{21-2\pi^2}{12} >0\]
and for $n \ge 1$ we get
\[ \begin{split}
G(n) =\ & \frac1{2(n+1)^2(n+2)^2} - \sum_{k \ge n+2} \frac1{2k^2(k+1)^2}\\
=\ & \frac1{2(n+1)^2(n+2)^2} - \frac1{2(n+2)^2 (n+3)^2} - \frac1{2(n+3)^2(n+4)^2} - \sum_{k \ge n+4} \frac1{2k^2(k+1)^2}\\
=\ & \frac{-n^4-2n^3+27n^2+116n+124}{2(n+1)^2(n+2)^2(n+3)^2(n+4)^2} - \sum_{k \ge n+4} \frac1{2k^2(k+1)^2}.
\end{split}\]
One can show that $-n^4-2n^3+27n^2+116n+124 < 0$ for all $n \ge 7$. Hence, $G(n) < 0$ for all $n \ge 7$. For $n=1, \ldots, 6$ we can compute $G(n)$ separately, which gives
\begin{align*}
 G(1)&=\frac{119 - 12 \pi^2}{72} >0,\\
 G(2)&=\frac{237 - 8 \pi^2}{144}>0,\\
 G(3)&=\frac{11843-1200 \pi^2}{7200}<0,\\
 G(4)&=\frac{5921-600\pi^2}{3600} <0,\\
 G(5)&=\frac{290131-29400 \pi^2}{176400}<0,\\
 G(6)&=\frac{1160549 - 117600\pi^2}{705600}<0.
\end{align*}
From the above data we can deduce that
$$\mathcal{M}=\bigcup_{\omega\in\{0,1\}^3}I_\omega,$$
so $\mathcal M$ consists of 8 closed intervals. A geometrical construction of $\mathcal{M}$ is plotted in Figure \ref{fig:4}.
\end{proof}

 \begin{center}
\begin{figure}[h!]

\begin{tikzpicture}[xscale=25,yscale=20,axis/.style={very thick, ->}, important line/.style={thick}, dashed line/.style={dashed, thin},
    pile/.style={thick, ->, >=stealth', shorten <=2pt, shorten>=2pt},every node/.style={color=black} ]

    \node[] at({0.25}, 0.01){$I_\epsilon$};

    %1th
    \draw[important line] (1-0.82246703342411320303284583133063, 0)--(0.82246703342411320303284583133063-1/2, 0);

    %2th
     \node[] at(1-0.81, -0.013){\tiny{$I_1$}}; \node[] at(1-0.685, -0.013){\tiny{$I_0$}};
     \draw[important line] (1-0.82246703342411320303284583133063, -0.02)--(0.82246703342411320303284583133063-1/2-1/8, -0.02);
     \draw[important line] (1-0.82246703342411320303284583133063+1/8, -0.02)--(0.82246703342411320303284583133063-1/2, -0.02);

     %3th
     \node[] at(1-0.82, -0.033){\tiny{$I_{11}$}}; \node[] at(1-0.80, -0.033){\tiny{$I_{10}$}};
     \node[] at(1-0.695, -0.033){\tiny{$I_{01}$}}; \node[] at(1-0.675, -0.033){\tiny{$I_{00}$}};
     \draw[important line] (1-0.82246703342411320303284583133063, -0.04)--(0.82246703342411320303284583133063-1/2-1/8-1/72, -0.04);
     \draw[important line](1-0.82246703342411320303284583133063+1/72,-0.04)--(0.82246703342411320303284583133063-1/2-1/8, -0.04);
     \draw[important line] (1-0.82246703342411320303284583133063+1/8, -0.04)--(0.82246703342411320303284583133063-1/2-1/72, -0.04);
     \draw[important line] (1-0.82246703342411320303284583133063+1/8+1/72, -0.04)--(0.82246703342411320303284583133063-1/2, -0.04);

    %4th
     \draw[important line](1-0.82246703342411320303284583133063,-0.06)--(0.82246703342411320303284583133063-1/2-1/8-1/72-1/288, -0.06);
     \draw[important line](1-0.82246703342411320303284583133063+1/288, -0.06)--(0.82246703342411320303284583133063-1/2-1/8-1/72, -0.06);
     \draw[important line](1-0.82246703342411320303284583133063+1/72,-0.06)--(0.82246703342411320303284583133063-1/2-1/8-1/288, -0.06);
     \draw[important line](1-0.82246703342411320303284583133063+1/72+1/288,-0.06)--(0.82246703342411320303284583133063-1/2-1/8, -0.06);
     \draw[important line] (1-0.82246703342411320303284583133063+1/8, -0.06)--(0.82246703342411320303284583133063-1/2-1/72-1/288, -0.06);
     \draw[important line] (1-0.82246703342411320303284583133063+1/8+1/288, -0.06)--(0.82246703342411320303284583133063-1/2-1/72, -0.06);
     \draw[important line] (1-0.82246703342411320303284583133063+1/8+1/72, -0.06)--(0.82246703342411320303284583133063-1/2-1/288, -0.06);
     \draw[important line] (1-0.82246703342411320303284583133063+1/8+1/72+1/288, -0.06)--(0.82246703342411320303284583133063-1/2, -0.06);

    %5th
      \draw[important line](1-0.82246703342411320303284583133063,-0.08)--(0.82246703342411320303284583133063-1/2-1/8-1/72-1/288, -0.08);
     \draw[important line](1-0.82246703342411320303284583133063+1/288, -0.08)--(0.82246703342411320303284583133063-1/2-1/8-1/72, -0.08);
     \draw[important line](1-0.82246703342411320303284583133063+1/72,-0.08)--(0.82246703342411320303284583133063-1/2-1/8-1/288, -0.08);
     \draw[important line](1-0.82246703342411320303284583133063+1/72+1/288,-0.08)--(0.82246703342411320303284583133063-1/2-1/8, -0.08);
     \draw[important line] (1-0.82246703342411320303284583133063+1/8, -0.08)--(0.82246703342411320303284583133063-1/2-1/72-1/288, -0.08);
     \draw[important line] (1-0.82246703342411320303284583133063+1/8+1/288, -0.08)--(0.82246703342411320303284583133063-1/2-1/72, -0.08);
     \draw[important line] (1-0.82246703342411320303284583133063+1/8+1/72, -0.08)--(0.82246703342411320303284583133063-1/2-1/288, -0.08);
     \draw[important line] (1-0.82246703342411320303284583133063+1/8+1/72+1/288, -0.08)--(0.82246703342411320303284583133063-1/2, -0.08);

     %6th
      %\draw[important line](1-0.82246703342411320303284583133063,-0.1)--(0.82246703342411320303284583133063-1/2-1/8-1/72-1/288, -0.1);
%     \draw[important line](1-0.82246703342411320303284583133063+1/288, -0.1)--(0.82246703342411320303284583133063-1/2-1/8-1/72, -0.1);
%     \draw[important line](1-0.82246703342411320303284583133063+1/72,-0.1)--(0.82246703342411320303284583133063-1/2-1/8-1/288, -0.1);
%     \draw[important line](1-0.82246703342411320303284583133063+1/72+1/288,-0.1)--(0.82246703342411320303284583133063-1/2-1/8, -0.1);
%     \draw[important line] (1-0.82246703342411320303284583133063+1/8, -0.1)--(0.82246703342411320303284583133063-1/2-1/72-1/288, -0.1);
%     \draw[important line] (1-0.82246703342411320303284583133063+1/8+1/288, -0.1)--(0.82246703342411320303284583133063-1/2-1/72, -0.1);
%     \draw[important line] (1-0.82246703342411320303284583133063+1/8+1/72, -0.1)--(0.82246703342411320303284583133063-1/2-1/288, -0.1);
%     \draw[important line] (1-0.82246703342411320303284583133063+1/8+1/72+1/288, -0.1)--(0.82246703342411320303284583133063-1/2, -0.1);

\end{tikzpicture}
\quad\quad
\begin{tikzpicture}[xscale=600,yscale=20,axis/.style={very thick, ->}, important line/.style={thick}, dashed line/.style={dashed, thin},
    pile/.style={thick, ->, >=stealth', shorten <=2pt, shorten>=2pt},every node/.style={color=black} ]

    %1th
    %\draw[important line] (1-0.82246703342411320303284583133063, 0)--(0.82246703342411320303284583133063-1/2, 0);

    %2th
     %\draw[important line] (1-0.82246703342411320303284583133063, -0.02)--(0.82246703342411320303284583133063-1/2-1/8, -0.02);
     %\draw[important line] (1-0.82246703342411320303284583133063+1/8, -0.02)--(0.82246703342411320303284583133063-1/2, -0.02);

    \node[] at({1-0.82246703342411320303284583133063+1/300}, -0.03){$I_{11}$};
     %3th
     \draw[important line] (1-0.82246703342411320303284583133063, -0.04)--(0.82246703342411320303284583133063-1/2-1/8-1/72, -0.04);
     %\draw[important line](1-0.82246703342411320303284583133063+1/72,-0.04)--(0.82246703342411320303284583133063-1/2-1/8, -0.04);
    %4th
     \draw[important line](1-0.82246703342411320303284583133063,-0.06)--(0.82246703342411320303284583133063-1/2-1/8-1/72-1/288, -0.06);
     \draw[important line](1-0.82246703342411320303284583133063+1/288, -0.06)--(0.82246703342411320303284583133063-1/2-1/8-1/72, -0.06);

    %5th
      \draw[important line](1-0.82246703342411320303284583133063,-0.08)--(0.82246703342411320303284583133063-1/2-1/8-1/72-1/288-1/800, -0.08);
      {\draw[important line](1-0.82246703342411320303284583133063+1/800,-0.082)--(0.82246703342411320303284583133063-1/2-1/8-1/72-1/288, -0.082);}
      \draw[important line](1-0.82246703342411320303284583133063+1/288, -0.08)--(0.82246703342411320303284583133063-1/2-1/8-1/72-1/800, -0.08);
      {\draw[important line](1-0.82246703342411320303284583133063+1/288+1/800, -0.082)--(0.82246703342411320303284583133063-1/2-1/8-1/72, -0.082);}

     %6th
      \draw[important line](1-0.82246703342411320303284583133063,-0.1)--(0.82246703342411320303284583133063-1/2-1/8-1/72-1/288-1/800-1/1800, -0.1);
      \draw[important line](1-0.82246703342411320303284583133063+1/1800,-0.102)--(0.82246703342411320303284583133063-1/2-1/8-1/72-1/288-1/800, -0.102);
      {\draw[important line](1-0.82246703342411320303284583133063+1/800,-0.1)--(0.82246703342411320303284583133063-1/2-1/8-1/72-1/288-1/1800, -0.1);
      \draw[important line](1-0.82246703342411320303284583133063+1/800+1/1800,-0.102)--(0.82246703342411320303284583133063-1/2-1/8-1/72-1/288, -0.102);}
      \draw[important line](1-0.82246703342411320303284583133063+1/288, -0.1)--(0.82246703342411320303284583133063-1/2-1/8-1/72-1/800-1/1800, -0.1);
      \draw[important line](1-0.82246703342411320303284583133063+1/288+1/1800, -0.102)--(0.82246703342411320303284583133063-1/2-1/8-1/72-1/800, -0.102);
      {\draw[important line](1-0.82246703342411320303284583133063+1/288+1/800, -0.1)--(0.82246703342411320303284583133063-1/2-1/8-1/72-1/1800, -0.1);
      \draw[important line](1-0.82246703342411320303284583133063+1/288+1/800+1/1800, -0.102)--(0.82246703342411320303284583133063-1/2-1/8-1/72, -0.102);}

      %7th
      \draw[important line](1-0.82246703342411320303284583133063,-0.12)--(0.82246703342411320303284583133063-1/2-1/8-1/72-1/288-1/800-1/1800-1/3528, -0.12);
      \draw[important line](1-0.82246703342411320303284583133063+1/3528,-0.122)--(0.82246703342411320303284583133063-1/2-1/8-1/72-1/288-1/800-1/1800, -0.122);
      \draw[important line](1-0.82246703342411320303284583133063+1/1800,-0.12)--(0.82246703342411320303284583133063-1/2-1/8-1/72-1/288-1/800-1/3528, -0.12);
      \draw[important line](1-0.82246703342411320303284583133063+1/1800+1/3528,-0.122)--(0.82246703342411320303284583133063-1/2-1/8-1/72-1/288-1/800, -0.122);
      {\draw[important line](1-0.82246703342411320303284583133063+1/800,-0.12)--(0.82246703342411320303284583133063-1/2-1/8-1/72-1/288-1/1800-1/3528, -0.12);
      \draw[important line](1-0.82246703342411320303284583133063+1/800+1/3528,-0.122)--(0.82246703342411320303284583133063-1/2-1/8-1/72-1/288-1/1800, -0.122);}
      {\draw[important line](1-0.82246703342411320303284583133063+1/800+1/1800,-0.12)--(0.82246703342411320303284583133063-1/2-1/8-1/72-1/288-1/3528, -0.12);
      \draw[important line](1-0.82246703342411320303284583133063+1/800+1/1800+1/3528,-0.122)--(0.82246703342411320303284583133063-1/2-1/8-1/72-1/288, -0.122);}

      \draw[important line](1-0.82246703342411320303284583133063+1/288, -0.12)--(0.82246703342411320303284583133063-1/2-1/8-1/72-1/800-1/1800-1/3528, -0.12);
       \draw[important line](1-0.82246703342411320303284583133063+1/288+1/3528, -0.122)--(0.82246703342411320303284583133063-1/2-1/8-1/72-1/800-1/1800, -0.122);
      \draw[important line](1-0.82246703342411320303284583133063+1/288+1/1800, -0.12)--(0.82246703342411320303284583133063-1/2-1/8-1/72-1/800-1/3528, -0.12);
      \draw[important line](1-0.82246703342411320303284583133063+1/288+1/1800+1/3528, -0.122)--(0.82246703342411320303284583133063-1/2-1/8-1/72-1/800, -0.122);

      {\draw[important line](1-0.82246703342411320303284583133063+1/288+1/800, -0.12)--(0.82246703342411320303284583133063-1/2-1/8-1/72-1/1800-1/3528, -0.12);
      \draw[important line](1-0.82246703342411320303284583133063+1/288+1/800+1/3528, -0.122)--(0.82246703342411320303284583133063-1/2-1/8-1/72-1/1800, -0.122);}
      {\draw[important line](1-0.82246703342411320303284583133063+1/288+1/800+1/1800, -0.12)--(0.82246703342411320303284583133063-1/2-1/8-1/72-1/3528, -0.12);
      \draw[important line](1-0.82246703342411320303284583133063+1/288+1/800+1/1800+1/3528, -0.122)--(0.82246703342411320303284583133063-1/2-1/8-1/72, -0.122);}

     %8th
     %\draw[important line](1-0.82246703342411320303284583133063,-0.14)--(0.82246703342411320303284583133063-1/2-1/8-1/72-1/288-1/800-1/1800-1/3528-1/6272, -0.14);
%     \draw[important line](1-0.82246703342411320303284583133063+1/6272,-0.142)--(0.82246703342411320303284583133063-1/2-1/8-1/72-1/288-1/800-1/1800-1/3528, -0.142);
\end{tikzpicture}
\caption{Left: the first four levels of $\mathcal{M}$. Right: the first four levels of $I_{11}$.}
\label{fig:4}
\end{figure}
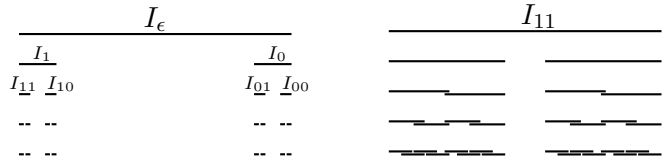
\end{center}

We can now prove Theorem~\ref{thm:main-3}.

\begin{proof}[Proof of Theorem~\ref{thm:main-3}]
Proposition~\ref{p:luroth8} gives the first part of the theorem. For the second part, we would like to use Proposition~\ref{p:overlaps}. Let $\om \in \{0,1\}^n$. By Lemma~\ref{le:M-o-0-leq-M-o-1} this is given by
\[ \frac{I(n-1)}{I(n)} \in (1,\infty).\]
By \eqref{q:sizeI} we obtain that
\[ I(n) = \int_{[0,1]} F_{\omega  \bar 1} - F_{\omega  \bar 0} \, d\lambda = \sum_{k \ge n+1} \frac1{2k^2(k+1)^2}.\]
One can check that
\[ \lim_{n \to \infty} \frac{I(n-1)}{I(n)} =  1+\lim_{n \to \infty}\frac1{n^2(n+1)^2} \left( \sum_{k \ge n+1} \frac1{k^2(k+1)^2} \right)^{-1} =1.\]
The result then follows from Proposition~\ref{p:overlaps}.
\end{proof}

\medskip
We end by giving some examples of partitions for which $\lim_{n\to\infty}\rho_n$ does not exist to illustrate that in this case several different situations can occur.

\begin{example}
Let $\alpha=\{A_n=(t_{n+1},t_n]:n\in\mathbb{N}\} \cup \{ A_\infty = \{0\} \}$ be such that
\begin{align*}
\begin{split}
t_n:=\left\{
\begin{array}{ll}
\frac{1}{2^m}\quad & \text{for } n=2m+1 \text{ and } m\geq 0; \\
\frac{3}{5\cdot 2^m} \quad & \text{for } n=2(m+1)\text{ and } m\geq 0.\\
\end{array}
\right.
\end{split}
\end{align*}
Then $\rho_n=\frac{t_{n+1}}{t_n}=\frac{3}{5}$ for $n=2m+1$ and $\rho_n=\frac{t_{n+1}}{t_n}=\frac{5}{6}$ for $n=2(m+1)$, so that
$$\liminf_{n\to\infty} \frac{t_{n+1}}{t_n}=\frac{3}{5} \quad \text{ and } \quad \limsup_{n\to\infty} \frac{t_{n+1}}{t_n}=\frac{5}{6}.$$
In other words, the limit $\lim_{n\to\infty} \frac{t_{n+1}}{t_n}$ does not exist. By \eqref{eq:G-N-p-ge-1-2} for $n=2m$ we have
\begin{align*}
G(n)=&\left(\frac{5}{6}-\frac{9}{50}+\frac{9}{5}-\frac{3}{2}\right)\frac{1}{2^m}+
\left(\frac{25}{72}-\frac{3}{5}\right)\sum_{k=m}^\infty\frac{3}{5\cdot2^k}+\left(\frac{9}{50}-\frac{5}{6}\right)\sum_{k=m+1}^\infty\frac{1}{2^k}\\
=&\left(\frac{9}{5}-\frac{3}{2}+\frac{5}{12}-\frac{18}{25}\right)\frac{1}{2^m}=-\frac{1}{300\cdot 2^m}<0.
\end{align*}
Similarly, for $n=2m+1$ we find
$$G(n)=-\frac{4}{75\cdot 2^m}<0.$$
Hence, in this case $\mathcal M = [M_{\bar 1}, M_{\bar 0}]$.
\end{example}

\begin{example}
Let $\alpha=\{A_n=(t_{n+1},t_n]:n\in\mathbb{N}\} \cup \{ A_\infty = \{0\} \}$ be such that
\begin{align*}
\begin{split}
t_n:=\left\{
\begin{array}{ll}
\frac{1}{3^m}\quad & \text{for } n=2m+1 \text{ and } m\geq 0; \\
\frac{21}{40\cdot 3^m} \quad & \text{for } n=2(m+1)\text{ and } m\geq 0.\\
\end{array}
\right.
\end{split}
\end{align*}
Then $\rho_n=\frac{t_{n+1}}{t_n}=\frac{21}{40}$ for $n=2m+1$, and $\rho_n=\frac{t_{n+1}}{t_n}=\frac{40}{63}$ for $n=2(m+1)$. %Then $$\liminf_{n\to\infty} \frac{t_{n+1}}{t_n}=\frac{21}{40} \quad and \quad \limsup_{n\to\infty} \frac{t_{n+1}}{t_n}=\frac{40}{63},$$ the limit $\lim_{n\to\infty} \frac{t_{n+1}}{t_n}$ not exists.
By \eqref{eq:G-N-p-ge-1-2} for $n=2m$ we have
\begin{align*}
G(n)=&\left(\frac{20}{21}-\frac{21^2}{2\cdot 40^2}+\frac{63}{40}-\frac{3}{2}\right)\frac{1}{3^m}+
\left(\frac{40^2}{2\cdot63^2}-\frac{63}{80}\right)\sum_{k=m}^\infty\frac{21}{40\cdot3^k}\\
& +\left(\frac{21^2}{2\cdot40^2}-\frac{20}{21}\right)\sum_{k=m+1}^\infty\frac{1}{3^k}\\
=&\frac{841}{40320 \cdot 3^m}>0,
\end{align*}
while for $n=2m+1$ we get
$$G(n)=-\frac{12391}{302400\cdot3^m}<0.$$
In this case we cannot say what the structure of $\mathcal M$ will be, but we do note that even for $\frac{t_{n+1}}{t_n}>\frac{1}{2}$ one can have that $G(n)>0$.
\end{example}

\begin{example}
Let $\alpha=\{A_n=(t_{n+1},t_n]:n\in\mathbb{N}\} \cup \{ A_\infty = \{0\} \}$ be such that
\begin{align*}
\begin{split}
t_n:=\left\{
\begin{array}{ll}
\frac{1}{4^m}\quad & \text{for } n=2m+1 \text{ and } m\geq 0; \\
\frac{1}{3\cdot 4^m} \quad & \text{for } n=2(m+1)\text{ and } m\geq 0.\\
\end{array}
\right.
\end{split}
\end{align*}
Then $\rho_n=\frac{t_{n+1}}{t_n}=\frac{1}{3}$ for $n=2m+1$, and $\rho_n=\frac{t_{n+1}}{t_n}=\frac{3}{4}$ for $n=2(m+1)$. %Then $$\liminf_{n\to\infty} \frac{t_{n+1}}{t_n}=\frac{1}{3} \quad and \quad \limsup_{n\to\infty} \frac{t_{n+1}}{t_n}=\frac{3}{4},$$ the limit $\lim_{n\to\infty} \frac{t_{n+1}}{t_n}$ not exists.
By
\eqref{q:smallg}
and \eqref{eq:def-G}  for $n=2m$ we have
\begin{align*}
G(n)=&\frac{1}{2}\left(1-\frac{1}{3}-\frac{1}{9}\right)\frac{1}{4^m}-
\left(\frac{2}{3}-\frac{9}{32}+\frac{9}{8}-\frac{3}{2}\right)\sum_{k=m}^\infty\frac{1}{3\cdot4^k}
-\frac{1}{2}\left(1-\frac{1}{3}-\frac{1}{9}\right)\sum_{k=m+1}^\infty\frac{1}{4^k}\\
=&\left(\frac{5}{27}-\frac{1}{216}\right)\frac{1}{4^m}>0,
\end{align*}
%{\color{red}
%\begin{align*}
%G(n)=&\left(\frac32-\frac{1}{18} + 1-\frac32\right)\frac{1}{4^m}-
%\left(\frac{2}{3}-\frac{9}{32}\right)\sum_{k=m}^\infty\frac{1}{3\cdot4^k}
%-\left(\frac32-\frac{1}{18}\right)\sum_{k=m+1}^\infty\frac{1}{4^k}\\
%=&\frac{7}{24 \cdot 4^m}>0,
%\end{align*}
%}
while for $n=2m+1$ we get
\begin{align*}
G(n)=&\left(\frac{2}{3}-\frac{9}{32}+\frac{9}{8}-\frac{3}{2}\right)\frac{1}{3\cdot4^m}-\frac{1}{2}\left(1-\frac{1}{3}-\frac{1}{9}\right)\sum_{k=m+1}^\infty\frac{1}{4^k}\\
&-\left(\frac{2}{3}-\frac{9}{32}+\frac{9}{8}-\frac{3}{2}\right)\sum_{k=m+1}^\infty\frac{1}{3\cdot4^k}\\
=\ & \left(\frac{1}{432}-\frac{5}{54}\right)\frac{1}{4^m}<0.
\end{align*}
%{\color{red}
%\begin{align*}
%G(n)=\ &\left(\frac12-\frac34-\frac{9}{32} \right)\frac{1}{3\cdot4^m}+\frac{1}{2}\cdot \frac13 \sum_{k=m+1}^\infty\frac{1}{4^k} +\frac12\cdot \frac34 \sum_{k=m+1}^\infty\frac{1}{3\cdot4^k}\\
%=\ & -\frac{23}{288 \cdot 4^m}<0.
%\end{align*}
%}
Also in this case it is not clear what the structure of $\mathcal M$ is.
\end{example}

\section*{Acknowledgement}
The authors would like to thank Derong Kong for valuable discussions and helpful suggestions and the Mathematical Institute of Leiden University for their hospitality. The first author was supported by the China Scholarship Council and NSFC No.~11971079.

\bibliographystyle{abbrv}
\bibliography{Fractal-Expansions}

\end{document}